\documentclass[11pt,a4paper]{amsart}

\usepackage[english]{babel}

\usepackage[utf8]{inputenc} 	% allow utf-8 input
\usepackage[T1]{fontenc}    	% use 8-bit T1 fonts

\usepackage{lmodern}

\usepackage{amssymb,amsmath}
            
\usepackage{url}            		% URL typesetting
\usepackage{amsfonts}       	% blackboard math symbols
\usepackage{nicefrac}       	% compact symbols for 1/2, etc.
\usepackage{microtype}      	% Micro typography
\usepackage{lipsum}			% Can be removed after putting your text content
\usepackage{hyperref} 		% For hyperlinks in the PDF
\usepackage{upgreek}		% Better notation of the second fundamental form
\usepackage{graphbox}		% Center align figures
\usepackage[shortlabels]{enumitem} 	% Customised lists
\usepackage[alphabetic, abbrev]{amsrefs}  %Bibliography
\usepackage{stackengine}
\usepackage{graphicx}
\usepackage{subcaption}
\usepackage{tikz}
\usepackage{mwe}
\usepackage{float}
\usepackage[toc,page]{appendix}
\usepackage{mathtools}
\usepackage{stmaryrd}
\usepackage{faktor}

\usepackage[bottom=3cm, top=3cm, left=3.5cm, right=3.5cm]{geometry}

\theoremstyle{plain}
\newtheorem{theorem}{Theorem}[section]
\newtheorem{proposition}[theorem]{Proposition}
\newtheorem{lemma}[theorem]{Lemma}
\newtheorem{corollary}[theorem]{Corollary}

\theoremstyle{definition} 
\newtheorem{definition}[theorem]{Definition}

\theoremstyle{remark}
\newtheorem{remark}[theorem]{Remark}

\hypersetup{
    colorlinks,
    linkcolor={red!50!black},
    citecolor={blue!50!black},
    urlcolor={blue!80!black}}
\newcommand{\RomanNumeralCaps}[1] {\MakeUppercase{\romannumeral #1}}

\newcommand{\R}{\mathbb{R}}	
\newcommand{\vol}{\operatorname{vol}}
\newcommand{\II}{\operatorname{ \RomanNumeralCaps{2}}}

\newcommand{\spa}{\mathrm{span}}

\newcommand{\hess}{\operatorname{Hess}}
\newcommand{\distr}{D}

\newcommand{\tb}{\mathrm{tb}}

\newcommand{\tr}{\operatorname{tr}}
\newcommand{\dive}{\operatorname{div}}	
\newcommand{\eps}{\varepsilon}
	
\newcommand{\sign}{\operatorname{sign}}
\newcommand{\area}{\operatorname{Area}}
\newcommand{\contact}{\omega}

\setlist[itemize]{noitemsep} 	% Make itemize lists more compact

\providecommand{\keyword}[1]
{
  \small	
  \textbf{Keywords:} #1
}

\newcommand{\mysetminusD}{\hbox{\tikz{\draw[line width=0.6pt, line cap=round] (3pt,0) -- (0,6pt);}}}
\newcommand{\mysetminusT}{\mysetminusD}
\newcommand{\mysetminusS}{\hbox{\tikz{\draw[line width=0.45pt, line cap=round] (2pt,0) -- (0,4pt);}}}
\newcommand{\mysetminusSS}{\hbox{\tikz{\draw[line width=0.4pt, line cap=round] (1.5pt,0) -- (0,3pt);}}}
\newcommand{\mysetminus}{\mathbin{\mathchoice{\mysetminusD}{\mysetminusT}{\mysetminusS}{\mysetminusSS}}}		

\title[Induced distance on surfaces in 3D contact sub-Riemannian manifolds]{On the induced geometry on surfaces in\\3D contact sub-Riemannian manifolds}

\date{\today}

\author{Davide Barilari}
\address{Davide Barilari, Dipartimento di Matematica "Tullio Levi-Civita", Università di Padova, Via Trieste 63, Padova, Italy.}
\email{barilari@math.unipd.it}

\author{Ugo Boscain}
\address{Ugo Boscain, CNRS, Laboratoire Jacques-Louis Lions, team Inria CAGE, Universit\'e de Paris, Sorbonne Universit\'e boîte courrier 187, 75252 Paris Cedex 05 Paris France}
\email{ugo.boscain@upmc.fr}

\author{Daniele Cannarsa}
\address{Daniele Cannarsa, Universit\'e de Paris, Sorbonne Universit\'e, CNRS, Inria, Institut de Math\'ematiques de Jussieu-Paris Rive Gauche, F-75013 Paris, France}
\email{daniele.cannarsa@imj-prg.fr}

\begin{document}

\maketitle

%% %% %% %% %% %% %% %% %% %% %% %% %% %% %% %% %% %% %

\begin{abstract}
Given a surface $S$ in a 3D contact sub-Riemannian manifold $M$, we investigate the metric structure induced on $S$ by $M$, in the sense of length spaces.
 First,  we define a coefficient $\widehat K$ at characteristic points that determines locally the characteristic foliation of $S$.
Next, we identify some global conditions for the induced distance to be finite. 
In particular, we prove that the  induced distance is finite for surfaces with the topology of a sphere embedded in a tight  coorientable distribution,  with isolated characteristic points.
\end{abstract}
\bigskip
 \keyword{contact geometry, sub-Riemannian geometry, length space, Riemannian approximation, Gaussian curvature, Heisenberg group.}
 
\setcounter{tocdepth}{1}

\tableofcontents

%% %% %% %% %% %% %% %% %% %% %% %% %% %% %% %% %% %% %
%	SECTION 1 --- INTRODUCTION
%% %% %% %% %% %% %% %% %% %% %% %% %% %% %% %% %% %% %

\section{Introduction}

	The study of the geometry of submanifolds $S$ of an ambient manifold $M$ with a given geometric structure is a classical subject.
	A familiar example, whose study goes back to Gauss, is that of a surface $S$ embedded in the Euclidean space $\R^3$.
In such case, $S$ inherits its natural Riemannian structure by restricting the metric tensor to the tangent space of $S$. 
The distance induced on $S$ by this metric tensor is not the restriction of the distance of $\R^3$ to points on $S$, but rather the length space structure induced on $S$ by the ambient space.

	Things are less straightforward for a smooth 3-manifold $M$ endowed with a contact sub-Riemannian structure $(D, g)$; here $D$ is a smooth contact distribution and $g$ is a smooth metric on it.
Indeed, for a two-dimensional submanifold $S$, the intersection $T_{x}S\cap D_{x}$ is one-dimensional for most points $x$ in $S$; thus, $TS\cap D$ is not a bracket-generating distribution and there is no well-defined sub-Riemannian  distance induced by $(M, \distr, g)$ on $S$. 
This fact is indeed more general, as already observed in \cite[Sec.~0.6.B]{Gromov1996}.

	Nevertheless, one can still define a distance on $S$ following the length space viewpoint: the sub-Riemannian distance $d_{sR}$ defines
the length of any continuous curve $\gamma: [0, 1] \to M$ as
\[
L_{sR}(\gamma) = \sup \left\{ {\sum}_{i=1}^{N} d_{sR}(\gamma(t_{i}),\gamma(t_{i+1})) \mid 0= t_{0}\leq \ldots\leq t_{N}=1 \right\},
\]
and one can define $d_{S}:S\times S \to  [0,+\infty]$ with
\[
d_{S}(x, y) = \inf \{L_{sR}(\gamma) \mid \gamma: [0, 1] \to S,~\gamma(0)=x,~\gamma(1)=y \}.
\]
The space $(S, d_{S})$ is called a \textit{length space}, and $d_{S}$ the \textit{induced distance} defined by $(M,d_{sR})$.
(In the theory of length metric spaces,  the induced distance $d_{S}$ is called intrinsic distance, emphasising that it depends uniquely on lengths of curves in $S$, see~\cite{BBI01}.) We stress that the induced distance $d_{S}$ is not the restriction $d_{sR}|_{S\times S}$  of the sub-Riemannian distance to $S$.
\\ 

	This paper studies necessary and sufficient conditions  on the surface $S$ for which the induced distance  $d_{S}$ is finite. i.e., $d_{S}(x, y) < +\infty$ for all points $x,y$ in $S$; this is equivalent to $(S,d_{S})$ being a metric space. 
	In the following lines we rephrase this property through the characteristic foliation of $S$. 
	
Recall that a curve $\bar \gamma$ is horizontal with respect to $D$ if it is Lipschitz, and its derivative $\dot{\bar\gamma}$ is in $D$ whenever defined.
Consider a continuous curve $\gamma:[0,1]\to S$. 
Its length is finite, i.e., $L_{sR}(\gamma)<+\infty$, if and only if  $\gamma$ is a reparametrisation of a curve $\bar \gamma$ horizontal with respect to $ D$; in such case, the length of $\gamma$ coincides with the sub-Riemannian length of $\bar\gamma$, i.e., the integral of $|\dot{\bar\gamma}|_{g}$.
We refer to \cite[Ch.~2]{BBI01} and \cite[Sec.~3.3]{ABB19} for more details.
Therefore, the distance $d_{S}(x,y)$ between two points $x$ and $y$ in $S$ is finite if and only if there exists a finite-length horizontal curve in $S$ with respect to $ D$ connecting the points $x$ and $y$.

	A point $p$ in $S$ is a \textit{characteristic point} if the tangent space $T_{p}S$ coincides with the distribution $D_{p}$.
The set  of characteristic points of $S$ is the \textit{characteristic set}, noted $\Sigma(S)$.
The characteristic set is closed due to the lower semi-continuity of the rank, and it cannot contain open sets since $D$ is bracket-generating.
Moreover, since the distribution $D$ is contact and $S$ is $C^{2}$, the set $\Sigma(S)$ is contained in a 1-dimensional submanifold of~$S$ (see Lemma~\ref{Rectifi-Sigma-S}) and, generically, it is composed of isolated points (see \cite[Par.~4.6]{geiges2008introduction}).

	Outside of the characteristic set, the intersection $TS\cap D$ is a one-dimensional distribution and  defines a regular one-dimensional foliation on $S\mysetminus\Sigma(S)$. 
This foliation extends to a singular foliation of $S$ by adding a singleton at every characteristic point. The resulting foliation is the \textit{characteristic foliation} of $S$. 
Note that any horizontal curve contained in $S\mysetminus\Sigma(S)$ stays inside a single one-dimensional leaf of the characteristic foliation. 

	In conclusion, the finiteness of $d_{S}$ is equivalent to the existence, for any two points in $S$, of a finite-length continuous concatenation of leaves of the characteristic foliation of $S$ connecting these two points. 

\begin{figure}[b!] \begin{center}
	\includegraphics[width=0.37\textwidth]{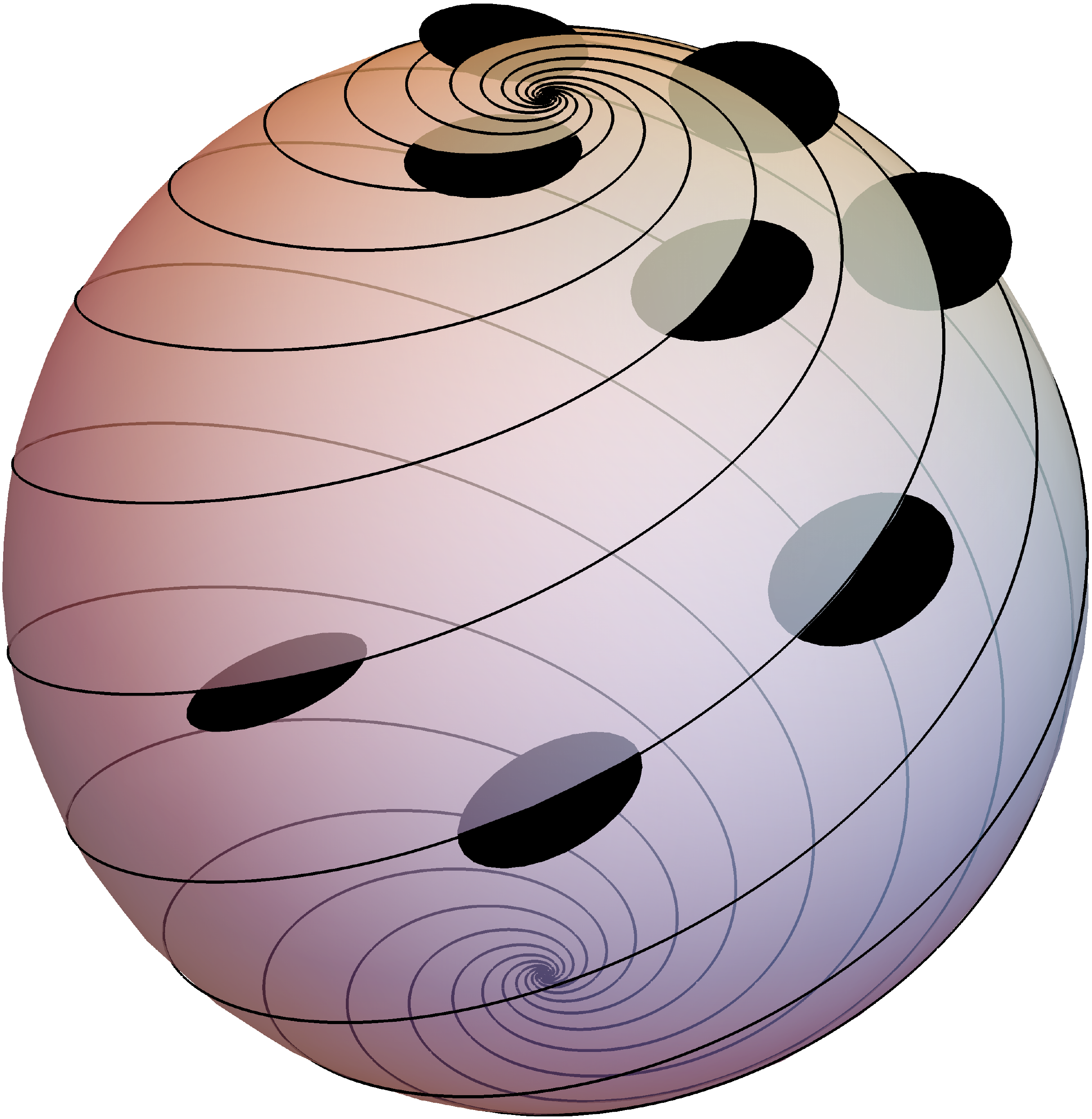}
	\caption{\label{img:sphere} \small 
		The characteristic foliation defined by the Heisenberg distribution $(\R^{3}, \ker(dz+\frac{1}{2}(y dx -x dy ))$ on an Euclidean sphere centred at the origin: any horizontal curve connecting points on different spirals goes though one of the characteristic points, at the North or the South pole.
		The sub-Riemannian length of the leaves spiralling around the characteristic points is finite because of Proposition~\ref{prop:leavesFiniteLength}. Thus, the induced distance $d_{S}$ is finite: this is a particular case of Theorem~\ref{thm:FiniteMetricSph}.
			}
\end{center}\end{figure}

%% %% %% %% %% %% %% %% %% %% %% %% %% %% %% %% %% %% %

\subsection{Main results}

	In this paper we prove two kind of results: local and global. 
On the local side, we are interested in the behaviour of the characteristic foliation around the characteristic points.
First, we use the Riemannian approximations of the sub-Riemannian space to associate with each characteristic point a real number.
Precisely, let $X_{0}$ be a vector field transverse to the distribution $D$ in a neighbourhood of a characteristic point $p\in \Sigma(S)$. 
Let $g^{X_{0}}$ be the Riemannian extension of $g$ for which
	\[
	\langle X_{0},D\rangle_{g^{X_{0}}}=0, \qquad |X_{0}|_{ g^{X_{0}}}=1.
	\]
The Riemannian metrics  $g^{\eps X_{0}}$, for $\eps>0$, are the \emph{Riemannian approximations} of $(D,g)$ with respect to $X_{0}$.	
Let $K^{X_{0}}$ be the Gaussian curvature of~$S$ with respect to $ g^{X_{0}}$, and let $B^{X_{0}}$ be the bilinear form $B^{X_{0}}:D\times D \to \R$ defined by 
	\[ 
	B^{X_{0}}(X,Y)=\alpha \qquad \mbox{if}\qquad [X,Y] = \alpha X_{0} \mod D.
	\]
Since $D$ is endowed with the metric $g$, the  bilinear form $B^{X_{0}}$ admits a well-defined determinant.  

\begin{theorem} \label{t:kconverge} 
Let $S$ be a $C^{2}$ surface embedded in a 3D contact sub-Riemannian manifold.
Let $p$ be a characteristic point of $S$, and let $X_{0}$ be a vector field transverse to the distribution $D$ in a neighbourhood of  $p$.
Then, in the notations defined above, the limit
	\begin{equation}  \label{defn:KS-as-limit}
	\widehat K_{p} = \lim_{\eps\to0}\frac{K_{p}^{\eps X_{0}}}{~\det B_{p}^{\eps X_{0}}~}
	\end{equation}
is finite and independent on the vector field $X_{0}$. 
\end{theorem}

	As we shall see, the coefficient $\widehat{K}_{p}$ determines the qualitative behaviour of the characteristic foliation near a characteristic point $p$.  
Given an open set $U$ in $S$, a  vector field  $X$ of class $C^{1}$ is a \textit{characteristic vector field} of $S$ in $U$ if, for all $x$ in $U$,
	\begin{equation} \label{eq:tangency-cvf}
	\spa_{\R} \, X(x) = \begin{cases} 
	\{0\},   & \mbox{if} \ x\in \Sigma(S),  \\
	\ T_{x}S\cap D_{x}, &  \mbox{otherwise},  \end{cases}
	\end{equation}
and satisfies the condition
	\begin{equation}\label{eq:non-deneric-cvf} 
	\dive X(p)\neq 0, \qquad \forall\, p\in \Sigma(S)\cap U.
	\end{equation}
 Notice that  $\dive X(p)$ is well-defined since $X(p)=0$, i.e., $p$ is a characteristic point, and it is independent on the volume form; in particular $\dive X(p)=\tr DX(p)$.
Due to Lemma~\ref{lem:existance-cvf}, one can show that locally there always exists a characteristic vector field, and that two characteristic vector fields are multiples by an everywhere non-zero function; in particular, if $X$ is a characteristic vector field, then also $-X$ it a characteristic vector field.
Finally, condition~(\ref{eq:tangency-cvf}) implies that  the characteristic foliation of $S$ in $U$ is the set of orbits of the dynamical system defined by $X$, and that the characteristic points are precisely the zeros of $X$, i.e., equilibrium points.

	Following the terminology of contact geometry (cf.\  for instance \cite[Par.~4.6]{geiges2008introduction}), given a characteristic point $p\in \Sigma(S)$ and a characteristic vector field $X$, the point $p$ is  \textit{elliptic} if $\det DX(p)>0$, and  \textit{hyperbolic} if $\det DX(p)<0$. 

\begin{proposition}\label{prop:Khat-and-eigenvalues}
Let $S$ be a $C^{2}$ surface embedded in a 3D contact sub-Riemannian manifold. 
Given a characteristic point $p$ in $\Sigma(S)$, let $X$ be a characteristic vector field $X$  near $p$. 
Then, $\tr DX(p)\neq 0$ and
	\begin{equation}\label{eq:Khat-and-DX(p)}
	\widehat K_{p} = -1+\frac{\det DX(p)}{~(\tr DX(p))^{2}}.
	\end{equation}
Thus, $p$ is hyperbolic if and only if $\widehat K_{p}<-1$, and $p$ it is elliptic  if and only if $\widehat K_{p}>-1$. 	
\end{proposition} 

This equality links $\widehat K_{p}$ to the eigenvalues of $DX(p)$, which determine the qualitative behaviour of the characteristic foliation around the characteristic point $p$.
This relation is made explicit  in Corollary~\ref{cor:qualitative-char-foli-non-deg} for a non-degenerate characteristic point, and in Corollary~\ref{cor:qualitative-char-foli-deg} for a degenerate characteristic point. 
Moreover, equation \eqref{eq:Khat-and-DX(p)} shows that $\widehat K_{p}$ is independent on the sub-Riemannian metric, and depends only on the line field defined by $D$ on $S$.  
 \\

	Still about local properties, we prove that the one-dimensional leaves of the characteristic foliation of $S$ which converge to a characteristic point have finite length.
Precisely, let $\ell$ be a leaf of the characteristic foliation of $S$; we say that a point $p$ in $S$ is a \emph{limit point} of $\ell$ if there exists a point $x$ in $\ell$ and a characteristic vector field~$X$ of $S$ such that 
\begin{equation}\label{eq:convergence-to-point}
e^{tX}(x)\to p \qquad \mbox{for}~t\to +\infty,
\end{equation}
 where $e^{tX}$ is the flow of $X$.
In such case, we denote the semi-leaf $\ell^{+}_{X}(x)=\{e^{tX}(x) ~|~ t\geq0\}$.
 With the above definition, a leaf can have at most two limit points: one for each extremity.
Finally,  notice that a limit point of a leaf must be a zero of the corresponding characteristic vector field $X$, i.e., a characteristic point of $S$. 

\begin{proposition}
\label{prop:leavesFiniteLength} 
Let $S$ be a $C^{2}$ surface embedded in a 3D contact sub-Riemannian manifold, and let $p$ be a limit point of a one-dimensional leaf $\ell$.
Let $x\in \ell$, and $X$ be a characteristic vector field such that  $e^{tX}(x)\to p$ for $t\to +\infty$. Then, the length of $\ell^{+}_{X}(x)$ is finite.
\end{proposition}

This result is not surprising, and it is  a consequence of the sub-Riemannian structure being contact. Indeed, for a non-contact distribution this conclusion is false; for instance, in  \cite[Lem.~2.1]{ZZ95} the authors prove that the length of the semi-leaves of the characteristic foliation of a Martinet surface converging to an elliptic point is infinite.
\\

On the global side, we determine some conditions for the induced metric~$d_{S}$ to be finite  under the assumption that there exists a global characteristic vector field of $S$.
In such case, for a compact, connected surface $S$ with isolated characteristic points, we show that $d_{S}$ is finite in the absence of the following  classes of leaves  in the characteristic foliation of $S$: nontrivial recurrent trajectories, periodic trajectories, and sided contours; see Proposition~\ref{prop:dS-finite}. 
Note that if $S$ is orientable and the distribution $D$ is \textit{coorientable}, i.e., there exists a global contact form $\contact$ defining the distribution (cf.\ also \eqref{eq:contact-form}), then $S$ admits a global characteristic vector field; see Lemma~\ref{lem:existance-cvf}. 
Recall that a distribution is \textit{tight} if it does not admit an \textit{overtwisted disk}, i.e., an embedding of a disk with horizontal boundary such that the distribution does not twists along the boundary.

\begin{theorem} \label{thm:FiniteMetricSph}
Let $(M, \distr,g)$ be a tight coorientable sub-Riemannian contact structure, and let $S$ be a $C^{2}$ embedded surface with isolated characteristic points, homeomorphic to a sphere. Then the induced distance $d_{S}$ is finite.
\end{theorem}

We stress that having isolated characteristic points is a generic property for a surface in a contact manifold. Example~\ref{exmp:horizontal-torus} and Example~\ref{exmp:vertical-torus} in the Heisenberg distribution show that, if $S$ is not a topological sphere, then $S$  presents possibly nontrivial recurrent trajectories or periodic trajectories, cases in which $d_{S}$ is not finite .
Moreover, if one removes the hypothesis of the contact structure being tight, then a sphere $S$ might present a periodic trajectory, hence the induced distance $d_{S}$ would not be finite. 
The compactness hypothesis is also important,  as one can see in Example~\ref{exmp:planes}.

%% %% %% %% %% %% %% %% %% %% %% %% %% %% %% %% %% %% %

\subsection*{Previous literature} 
	Characteristic foliations of surfaces in 3D contact manifolds are studied in numerous references;  here we use notions contained in \cite{Giroux1991,Giroux:2000aa, zbMATH03915280} and we refer to \cite{Etnyre_2003,geiges2008introduction} for an introduction to the subject.  Moreover, for an introduction to sub-Riemannian geometry we refer to \cite{montgomery2002tour,rifford:hal-01131787,jean:hal-01137580,ABB19}.

	The use of the Riemannian approximation scheme to define sub-Riemannian geometric invariants is a well-known technique. 
For example, it had already been used  in \cite{Pauls:2004aa} to study the horizontal mean curvature in relation to the minimal surfaces in the Heisenberg group, whose integrability is discussed in \cite{Danielli:2012aa}.
For a general description of the properties of the Riemannian approximations in Heisenberg we also refer to \cite{Capogna:2007}.

	In this paper, we combine the Riemannian approximation scheme suitably normalised by the Lie bracket structure on the distribution to define the metric coefficient $\widehat K$ \emph{at the characteristic points}. 
Notice that usually in the literature the Riemannian approximation is employed  to define sub-Riemannian geometric invariants \emph{outside} of the characteristic set.
For instance, in~\cite{Balogh:2017aa} the authors  defined the sub-Riemannian Gaussian curvature at a point $x\in S\mysetminus \Sigma(S)$  as $\mathcal{K}_{S}(x)=\lim_{\eps\to0}K_{x}^{\eps X_{0}}$, and they proved that a Gauss-Bonnet type theorem holds; here the authors worked  in the setting of the Heisenberg group, and with $X_{0}$ equals to the Reeb vector field of the Heisenberg group.  
This construction is extended in \cite{Wang_2020} to the affine group and to the group of rigid motions of the Minkowski plane, and in \cite{Veloso:2020} to a general sub-Riemannian manifold. 
In the latter, the author linked  $\mathcal{K}_{S}$ with the curvature introduced in \cite{Diniz:2016aa}, and, when $\Sigma(S)=\emptyset$, they proved a Gauss-Bonnet theorem by  Stokes formula.
A Gauss-Bonnet theorem (in a different setting) was also proven in \cite{Agrachev_2008}.
We finally notice that the invariant $\mathcal{K}_{S}$ also appears in  \cite{Lee_2013}, where it is called curvature of transversality. 
An expression for $\mathcal{K}_{S}$ is provided also in \cite{articKaren20}, in relation to a new notion of stochastic processes in this setting.

%% %% %% %% %% %% %% %% %% %% %% %% %% %% %% %% %% %% %

\subsection*{Structure of the paper}
After some preliminaries contained in Section~\ref{sect:prelim},
in Section~\ref{sec:riemann-aprox} we  prove Theorem~\ref{t:kconverge}, by introducing the metric invariant $\widehat K$ defined at characteristic points. In Section~\ref{sec:local-study}, we write the metric invariant in terms of a characteristic vector field as in Proposition~\ref{prop:Khat-and-eigenvalues}, and we  study the length of the horizontal curves as in Proposition~\ref{prop:leavesFiniteLength}.
In Section~\ref{sect:topolog-skeleton}, we use the topological decomposition of a 2D flow  to prove Proposition~\ref{prop:dS-finite}, from which we deduce Theorem~\ref{thm:FiniteMetricSph} in Section~\ref{sect:sphere}.
Section~\ref{sect:examples} is devoted  some examples of induced distances  on surfaces in the Heisenberg group.

%% %% %% %% %% %% %% %% %% %% %% %% %% %% %% %% %% %% %

\subsection*{Acknowledgements}
We would like to thank Daniel Bennequin and Nicola Garofalo for  stimulating discussions.
This work was supported by the Grant ANR-15-CE40-0018 SRGI of the French ANR.
The third author is supported by the DIM Math Innov grant from R\'egion Île-de-France.

%% %% %% %% %% %% %% %% %% %% %% %% %% %% %% %% %% %% %
%	SECTION 2 --- PRELIMINARIES
%% %% %% %% %% %% %% %% %% %% %% %% %% %% %% %% %% %% %

\section{Preliminaries}\label{sect:prelim}

	In this paper, $M$ is a smooth 3-dimensional manifold, $(D,g)$  a smooth contact sub-Riemannian structure on $M$, and $S$ an embedded surface of class $C^{2}$. 
The contact distribution is, locally, the kernel of a contact form  $\contact   \in \Omega^{1}(M)$, which can be normalised to satisfy 
	\begin{equation}\label{eq:contact-form}
	D=\ker \contact, \qquad \contact \wedge d\contact\neq 0, \qquad d\contact|_{D}=\vol_{g}.
	\end{equation} 

Recall that a  point $p$ in $S$ is a characteristic point of $S$ if $T_{p}S=D_{p}$, and that the characteristic points  of $S$ form the \emph{characteristic set} $\Sigma(S)$. 
For $x\in S\mysetminus \Sigma(S)$, the intersection
	\begin{equation} \label{eq:gendi}
	 l_{x} = D_{x} \cap T_{x}S
	 \end{equation}
is one-dimensional, and we can think of \eqref{eq:gendi} as defining a generalised distribution $l$ in $S$ whose rank increases at characteristic points. Sometimes in the literature the (generalised) distribution $l$ is called the \emph{trace} of $D$ on $S$.
The distribution $l$ is not smooth at the characteristic points, hence it is more convenient to work with a characteristic vector field, that is a $C^{1}$ vector field of $S$ satisfying~\eqref{eq:tangency-cvf} and~\eqref{eq:non-deneric-cvf}.

\begin{lemma} \label{lem:existance-cvf} 
Assume that $S$ is orientable and that $D$ is coorientable. 
Then, $S$ admits a global characteristic vector field; moreover, the characteristic vector fields of $S$ are the vector fields $X$ for which there exists a volume form  $\Omega$ of~$S$ such that 
	\begin{equation} \label{eq:cvf-as-dual}
	\Omega(X,Y) =\contact(Y) \qquad \mbox{for all } \ Y \in T S.
	\end{equation}
\end{lemma}
 Indeed, formula \eqref{eq:cvf-as-dual} is the definition of characteristic vector field as given in  \cite[Par.~4.6]{geiges2008introduction}, meaning that the characteristic vector fields are dual to the contact form $\contact|_{S}$ with respect to the volume forms of $S$. In the previous reference it is shown that if a vector field satisfies~\eqref{eq:cvf-as-dual}, then it satisfies~\eqref{eq:tangency-cvf} and~\eqref{eq:non-deneric-cvf}.
Reciprocally, a vector field $\bar X$ satisfying~\eqref{eq:tangency-cvf} is a multiple of any vector field $X$ satisfying \eqref{eq:cvf-as-dual} for some function $\phi$ with $\phi|_{S\mysetminus \Sigma(S)}\neq 0$; additionally, if~\eqref{eq:non-deneric-cvf} holds, then $\phi|_{\Sigma(S)}\neq 0$; thus, $\bar X$ satisfies \eqref{eq:cvf-as-dual} with $\frac{1}{\phi}\Omega$ as volume form of $S$.

\begin{remark}\label{rmq:proportionality-cvf}
Since the volume forms of $S$ are proportional by nowhere-zero functions, the same holds for the characteristic vector fields.
\end{remark}

	Therefore, if the orientability hypotheses hold,  an equivalent definition of the characteristic foliation is the partition of $S$ into the orbits of a global characteristic vector field. 
This is a generalised foliation, as the dimension of the leaves is not constant since the characteristic set is partitioned in singletons.
\\

Let us provide another way to find, locally, an explicit expression for a local characteristic vector field.
Any point in $S$ admits a neighbourhood $U$ in $M$ in which there exists  an oriented orthonormal frame $(X_{1},X_{2})$  for $D|_{U}$, and a submersion $f$ of class $C^{2}$ for which $S$ is a level set, i.e.,  $S\cap U=f^{-1}(0)$ and $df|_{U}\neq 0$. 
In such case, a vector $V\in TM|_{U}$ is in $TS$ if and only if $Vf=0$; thus, for  a point $p\in U\cap S$,
\begin{equation}\label{eq:f1-f2-zero-in-Sigma}
	p\in \Sigma(S) \iff X_{1}f(p)=X_{2}f(p)=0.
\end{equation}
Moreover, since $[X_{2},X_{1}]_{p}\not\in D_{p}=T_{p}S$ at a characteristic point $p$, then $[X_{2},X_{1}]f(p)\neq 0$.

\begin{remark}\label{req:local-setting}
In the previous notation, the vector field $X_{f}$ defined by
	\begin{equation}\label{eq:c-vf}
	X_{f}= (X_{1}f)X_{2}-(X_{2}f)X_{1},
	\end{equation}
is a characteristic vector field of $S$. Indeed, it follows from the definition that, for all $x$ in $S$, the vector $X_{f}(x)$ is in $T_{x}S\cap D_{x}$, and that, due to \eqref{eq:f1-f2-zero-in-Sigma},  $X_{f}(p)=0$ if and only if $p\in \Sigma(S)$; thus, $X_{f}$ satisfies~\eqref{eq:tangency-cvf}.
Moreover, for all $p\in\Sigma(S)$,
\[
	\dive X_{f}(p)=X_{2}X_{1}f(p)-X_{1}X_{2}f(p)=[X_{2},X_{1}]f(p),
\] 
which is nonzero due to the contact condition; thus, $X_{f}$ satisfies~\eqref{eq:non-deneric-cvf}.
\end{remark}

\begin{lemma}\label{Rectifi-Sigma-S}
The characteristic set $\Sigma(S)$ of a  surface $S$ of class $C^{2}$ is contained in a 1-dimensional submanifold of S of class  $C^{1}$.
\end{lemma}
\begin{proof}
It suffices to show that for every point  $p$ in $\Sigma(S)$ there exists a neighbourhood $V$ of $p$ such that $V\cap \Sigma(S)$ is contained in an embedded $C^{1}$ curve. 
Let us fix a point  $p$ in $\Sigma(S)$, and a neighbourhood $U$ of $p$ in $M$ equipped with a frame $(X_{1},X_{2})$ and a function $f$ with the properties described above.
Because of \eqref{eq:f1-f2-zero-in-Sigma}, the characteristic points in $V=U\cap S$ are the solutions of the system $X_{1}f=X_{2}f=0$.

Due to the implicit function theorem, it suffices to show that  $d_{p}(X_{1}f)\neq 0$ or $d_{p}(X_{2}f)\neq0$.
Thanks to the contact condition, we have that $[X_{2},X_{1}]f(p)\neq 0$. As a consequence, since $X_{2}X_{1}f(p)=[X_{2},X_{1}]f(p)+X_{1}X_{2}f(p)$, at least one of the following is true: $X_{2}X_{1}f(p)\neq0$, or $X_{1}X_{2}f(p)\neq0$.
Assume that the first is true; then $d_{p}(X_{1}f)(X_{2})=X_{2}X_{1}f(p)\neq 0$. The other case being similar,  the lemma is proved. 
\end{proof}

For a more general discussion on the size of the characteristic set, we refer to \cite{Bal03} and references therein.

% Moreover, $d_{S}$ does not satisfies a ball-box theorem; therefore, it is not anisotopic in the sense of  \cite[Sec.~1.5]{Gromov1996}, and it is not equivalent to the distance one can define via \cite[Def.~1.1]{Nagel:1985aa} using the filtration inherited by $S$.

%% %% %% %% %% %% %% %% %% %% %% %% %% %% %% %% %% %% %
%	SECTION 3 --- RIEMANNIAN APPROXIMATIONS
%% %% %% %% %% %% %% %% %% %% %% %% %% %% %% %% %% %% %

 \section{Riemannian approximations and Gaussian curvature} \label{sec:riemann-aprox}
 
	In this section we discuss the Riemannian approximations of a sub-Riemannian structure, and we prove Theorem~\ref{t:kconverge} by using the asymptotic expansion of the Gaussian curvature $K^{\eps X_{0}}_{p}$ at a characteristic point $p$.
	\\

In order to define the metric coefficient $\widehat K_{p}$, one needs to  fix a vector field $X_{0}$  transverse to the distribution in a neighbourhood of $p$. If the distribution is coorientable, it is possible to make this choice globally.
 As described in the introduction, once this choice has been made, one can extend the sub-Riemannian metric $g$ to a family of Riemannian metrics $g^{\eps X_{0}}$ such that, for every $\eps>0$, one has $\langle D,X_{0}\rangle_{g^{\eps X_{0}}}=0$ and $|X_{0}|_{g^{\eps X_{0}}}= 1/\eps$. 
To simplify the notation, we drop the dependance from $X_{0}$ in the superscript, writing $g^{\eps}= g^{\eps X_{0}}$. 

Let $\overline{\nabla}^{\eps}$ be the Levi-Civita connection of $(M,g^{\eps})$.  
Since we study local properties, we can restrict to a domain equipped with an orthonormal oriented frame $(X_{1},X_{2})$ of $D$; thus, $(\eps X_{0},X_{1}, X_{2})$ is an  orthonormal basis of $g^{\eps}$. Due to the Koszul formula, one has
\[ 
	\big\langle\overline{\nabla}^{\eps}_{X_{i}}X_{j},X_{k} \big\rangle_{g^{\eps}} = \frac{1}{2}\Big( - \langle X_{i},[X_{j},X_{k}]\rangle_{g^{\eps}}+\langle X_{k},[X_{i},X_{j}] \rangle_{g^{\eps}}+\langle X_{j},[X_{k},X_{i}]\rangle_{g^{\eps}}\Big),
\]
for all $i,j,k=0,1,2$. This identity enables us to describe $\overline{\nabla}^{\eps}$ using the frame $( X_{0},X_{1}, X_{2})$, which is independent from $\eps$.
This is done using the Lie bracket structure of the frame, i.e., the $C^{\infty}$ functions $c^{k}_{ij}$  such that
\begin{equation} \label{StructtureCoeff}
	[X_{j}, X_{i}] = c^{1}_{ij}X_{1} + c^{2}_{ij}X_{2} + c^{0}_{ij}X_{0} \qquad \mbox{for } i,j=0,1,2.
\end{equation}
The functions $c^{k}_{ij}$ are  the \textit{structure constants} of the frame.

Thus, for every $\eps>0$,  we have that 
\begin{align}\label{eqs:nablaEps}
	\overline{\nabla}^{\eps}_{X_{i}}X_{i} 
		&= c_{i0}^{i}\eps^{2}~X_{0}+\frac{c_{i1}^{i}}{\eps^{2}}~X_{1}+\frac{c_{i2}^{i}}{\eps^{2}}~~X_{2}   &&i=0,1,2 \vphantom{\frac{1}{2}} \\
	\overline{\nabla}^{\eps}_{X_{j}}X_{i} 
		&=  \frac{1}{2}\left( -c_{0i}^{j}\eps^{2}-c_{0j}^{i}\eps^{2}+c_{ij}^{0}\right)X_{0} + c_{ij}^{j}X_{j} && i\neq j= 1,2\nonumber \\
	\overline{\nabla}^{\eps}_{X_{0}}X_{1} 
		&= -c_{01}^{0}X_{0} + \frac{1}{2}\Big( c_{02}^{1}-c_{01}^{2}+\frac{c_{12}^{0}}{\eps^{2}}\Big)X_{2}\nonumber\\
	\overline{\nabla}^{\eps}_{X_{0}}X_{2} 
		&= -c_{02}^{0}X_{0}+ \frac{1}{2}\Big(c_{01}^{2}-c_{02}^{1}-\frac{c_{12}^{0}}{\eps^{2}}\Big) X_{1},\nonumber
\end{align}
and the remaining derivatives $\overline{\nabla}^{\eps}_{X_{1}}X_{0} $ and $\overline{\nabla}^{\eps}_{X_{2}}X_{0} $ are computed using  that the connection is torsion-free. \\ 

Given the surface $S$, the second fundamental form $\II^{\eps}$ of $S$ is the projection of the Levi-Civita connection on the orthogonal to the tangent space of the  surface.
The Gaussian curvature $K^{\eps}=K^{\eps X_{0}}$ of $S$ in  $(M,g^{\eps})$ is defined by the Gauss formula
	\begin{equation}\label{eq:curvatureEmbed} 
	K^{\eps}=  K_{\mathrm{ext}}^{\eps} + \det (\II^{\eps}),
	\end{equation}
where, given a frame $(X,Y)$ of $TS$,  the extrinsic curvature $K_{\mathrm{ext}}^{\eps}$ is
	\begin{equation}\label{defn:extrinsic-curvature}
	K_{\mathrm{ext}}^{\eps} = \frac{\big\langle \overline\nabla^{\eps}_{X}\overline\nabla^{\eps}_{Y}Y-\overline\nabla^{\eps}_{Y}\overline\nabla^{\eps}_{X}Y-\overline\nabla^{\eps}_{[X,Y]}Y, X\big\rangle_{g^{\eps}}}
		{|X|_{g^{\eps}}^{2}|Y|_{g^{\eps}}^{2}-\langle X,Y \rangle_{g^{\eps}}^{2}},
	\end{equation}
and the determinant $\det \II^{\eps}$  of the second fundamental form  is 
	\begin{equation}\label{den:determinant-II}
	\det \II^{\eps} =  \frac{\big\langle \II^{\eps}(X,X),\II^{\eps}(Y,Y)\big\rangle_{g^{\eps}}-\big\langle \II^{\eps}(X,Y), \II^{\eps}(X,Y)\big\rangle_{g^{\eps}}}
		{|X|_{g^{\eps}}^{2}|Y|_{g^{\eps}}^{2}-\langle X,Y \rangle_{g^{\eps}}^{2}}. 
	\end{equation}
Both these quantities are independent on the frame $(X,Y)$ of $TS$ chosen to compute them. 

%% %% %% %% %% %% %% %% %% %% %% %% %% %% %% %% %% %% %
	
\subsection{Proof of Theorem~\ref{t:kconverge}}

To prove the theorem, we  explicitly compute  the asymptotic of the quantities in limit~(\ref{defn:KS-as-limit}). 
Let us fix a characteristic point $p$, and, in a neighbourhood of~$p$, let us fix an oriented orthonormal frame $(X_{1},X_{2})$ of $D$ and a submersion $f$ defining $S$.

The determinant of the bilinear form $B_{p}^{\eps X_{0}}$ is homogeneous in $\eps$, and satisfies
\begin{equation} \label{eq:expansion-denominator}
	\det B^{\eps X_{0}}_{p} =\frac{\det  B^{X_{0}}_{p}}{ \eps^{2} } = \frac{B_{p}^{X_{0}}(X_{1},X_{2})^{2} }{ \eps^{2} } =\frac{(c_{12}^{0}(p))^{2}}{\eps^{2}},
\end{equation}
where $c_{12}^{0}$ is defined in~(\ref{StructtureCoeff}).
Therefore, in order to prove the convergence of the limit in~(\ref{defn:KS-as-limit}), it suffices to show that the Gaussian curvature $K^{\eps X_{0}}_{p}$ at $p$ diverges at most as $1/\eps^{2}$.
	\\

Let us start with the computation of the determinant~(\ref{den:determinant-II}) of the second fundamental form at a characteristic point. 
It is convenient to  write the second fundamental form as 
	\[
	\II^{\eps}(X,Y)=\big \langle \overline\nabla^{\eps}_{X}Y,N^{\eps} \big\rangle N^{\eps}.
	\]
where $N^{\eps}$ is the Riemannian unitary gradient of $f$, i.e.,
\[
	N^{\eps}=\frac{(X_{1}f)X_{1}+(X_{2}f)X_{2}+\eps(X_{0}f)X_{0}}{\sqrt{(X_{1}f)^{2}+(X_{2}f)^{2}+\eps(X_{0}f)^{2}}}.
\]
At the characteristic point $p$, the gradient $N^{\eps}(p)$ simplifies to
	\begin{equation}\label{eq:normal-vector}
	N^{\eps}(p) = \eps\,\sign(X_{0}f)~X_{0}(p). 
	\end{equation}
To compute~(\ref{den:determinant-II}) one needs to choose a frame of $TS$;  we will use the frame $(F_{1},F_{2})$  with
\begin{equation}\label{eq:frame-TS}
	F_{i}=(X_{0}f)X_{i}-(X_{i}f)X_{0} \qquad \mbox{for } i=1,2.
\end{equation}
This frame is well-defined for $X_{0}f\neq 0$; in particular, it is suited to calculate the Gaussian curvature at the characteristic points. 
Recall that the horizontal Hessian of $f$ is 
\[
	\mathrm{Hess}_{H} (f) = 
		\begin{pmatrix}
		X_{1}X_{1}f & X_{1} X_{2} f\\
		X_{2}X_{1}f & X_{2} X_{2} f
		\end{pmatrix}.
\]

\begin{lemma}\label{lem:det-II}
	Let $p\in S$ be a characteristic point. Then, in the previous notations, for every $\eps>0$, the determinant~(\ref{den:determinant-II}) of the second fundamental form in $p$ is
	\[
	\det \II^{\eps}(p) =   \frac{1}{\eps^{2}}  \bigg(\frac{ \det \hess_{H} f(p)}{ (X_{0}f(p))^2  } - \frac{(c_{12}^{0}(p))^{2}}{4}  \bigg)
	+O(1).
	\]
\end{lemma}
\begin{proof}
Let $p$ be a characteristic point. Because $X_{1}f(p)=X_{2}f(p)=0$, one can show that, 
\[
	\overline{\nabla}^{\eps}_{F_{i}}F_{j} (p) =\Big( (X_{0}f)^{2} \overline{\nabla}^{\eps}_{X_{i}}X_{j} + (X_{0}f)(X_{i}X_{0}f)X_{j}-(X_{0}f)(X_{i}X_{j}f)X_{0}\Big)\Big|_{p},
\]
for $i,j=1,2$. 
Using formula~\eqref{eq:normal-vector} for $N^{\eps}$, one finds that only the component along $X_{0}$ plays a role in the second fundamental form in $p$. Thus, using the covariant derivatives in~(\ref{eqs:nablaEps}),
\[
	\langle \overline{\nabla}^{\eps}_{F_{i}}F_{j}, N^{\eps}\rangle\big|_{p}  
	=- \frac{|X_{0}f(p)|}{\eps}\bigg(  X_{i}X_{j}f + (X_{0}f)\frac{c_{ij}^{0}}{2}+(X_{0}f)\eps^{2}\frac{c_{0i}^{j}+ c_{0j}^{i}}{2} \bigg)\bigg|_{p},
\]
for $i, j=1,2$. This, together with $\big( |F_{1}|^{2}|F_{2}|^{2}-\langle F_{1}, F_{2}\rangle^{2}\big) \big|_{p}= (X_{0}f(p))^{4}$,
gives the result.
\end{proof}

Next, the extrinsic curvature~(\ref{defn:extrinsic-curvature})  is the sectional curvature of the plane $T_{p}S$ in $M$, which is known when $X_{0}$ is the Reeb vector field and $\eps=1$; this can be found for instance in \cite[Prop.~14]{barilari:2020aa}.
In our setting, the resulting expression for $\eps\to0$ is the following.

 \begin{lemma}\label{lem:ext-curvature}
	Let $p\in S$ be a characteristic point. Then, for every $\eps>0$, 
	\begin{equation*} \label{eq:c.ext-char}
		K_{\mathrm{ext}}^{\eps} (p)= -\frac{3}{4\eps^{2}} (c_{12}^{0}(p))^{2} +O(1).
	\end{equation*} 
\end{lemma}
\begin{proof}
	To compute the extrinsic curvature  we use  the frame $(X_{1},X_{2})$ of $TM$, which coincides with $T_{p}S=D_{p}$ at the characteristic point $p$.
	Then, to compute 
	\[
	K_{\mathrm{ext}}^{\eps}(p)=\langle \overline\nabla^{\eps}_{X_{1}}\overline\nabla^{\eps}_{X_{2}}X_{2}-\overline\nabla^{\eps}_{X_{2}}\overline\nabla^{\eps}_{X_{1}}X_{2}-\overline\nabla^{\eps}_{[X_{1},X_{2}]}X_{2}, X_{1}\rangle\Big|_{p}
	\]
 it suffices to use the expressions~(\ref{eqs:nablaEps}).
\end{proof}

\begin{remark} 
	Following the proof of Lemma~\ref{lem:det-II} and Lemma~\ref{lem:ext-curvature}, the exact expressions for $\det \II^{\eps}(p)$ and $K_{\mathrm{ext}}^{\eps}(p)$ at a characteristic point $p$ are, for all $\eps>0$,
	\begin{align*}
	\det \II^{\eps}(p) =&  
		+\frac{1}{\eps^{2}}\bigg( \frac{ \det \hess_{H} f}{ (X_{0} f)^2  } - \frac{(c_{12}^{0})^{2}}{4} \bigg) \Big|_{p}
	+ \eps^{2}  \bigg(c_{01}^{1} c_{02}^{2}-\frac{\big(c_{02}^{1}+c_{01}^{2}\big)^2}{4}\bigg)\Big|_{p} \\
	& 	 +\frac{1}{X_{0}f(p)}\bigg( c_{02}^{2} X_{1}X_{1}f +c_{01}^{1}X_{2}X_{2}f-\frac{c^{2}_{01} + c^{1}_{02}}{2}(X_{2}X_{1}f + X_{1}X_{2}f) \bigg)\Big|_{p},\\
	K_{\mathrm{ext}}^{\eps} (p)= &-\frac{3}{4}\frac{ (c_{12}^{0}(p))^2}{ \eps^{2}} -  \eps^{2} \bigg( c_{01}^{1} c_{02}^{2}-\frac{(c_{01}^{2}+c_{02}^{1})^2}{4} \bigg) \Big|_{p}\\
	&
	+\bigg( X_{2} (c_{12}^{1})  - X_{1}(c_{12}^{2})- (c_{12}^{1})^2 - (c_{12}^{2})^2+c_{12}^{0}  \frac{c_{01}^{2} - c_{02}^{1} }{2}\bigg) \Big|_{p}.
	\end{align*}
If one chooses as transversal vector field the Reeb vector field of the contact sub-Riemannian manifold, then one recognises the first and the second functional invariants of the sub-Riemannian structure, defined in \cite[Ch.~17]{ABB19}.
	Finally, notice that these expressions are still valid for non-contact distributions.
\end{remark}

\begin{proof}[Proof of Theorem~\ref{t:kconverge}] 
	In the previous notations, due to the Gauss formula~(\ref{eq:curvatureEmbed}), Lemma~\ref{lem:det-II} and Lemma~\ref{lem:ext-curvature}, the Gaussian curvature at a characteristic point $p$ satisfies 
	\[
	K^{\eps}_{p}=K^{\eps X_{0}} _{p}
		= \frac{(c_{12}^{0}(p))^{2}}{\eps^{2}}  \bigg(-1+\frac{ \det \hess_{H} f(p)}{ ~[X_{2}, X_{1}]f(p)^2  }  \bigg) + O(1).
	\]
Here we have used that $ c_{12}^{0}(p)X_{0}f(p) = [X_{2}, X_{1}]f(p)$ at $p$, which holds due to definition~\eqref{StructtureCoeff} and $X_{1}f(p)=X_{2}f(p)=0$.
Using formula~(\ref{eq:expansion-denominator}) for the determinant of $B^{\eps X_{0}}_{p}$, one finds that 	
\begin{equation}\label{eq:development-limit}
	\frac{K^{\eps X_{0}}_{p} }{\det B_{p}^{\eps X_{0}}} = -1+\frac{ \det \hess_{H} f(p)}{ [X_{2}, X_{1}]f(p)^2   } +O(\eps^{2}),
\end{equation}
which shows that the limit~(\ref{defn:KS-as-limit}) is finite. 
Moreover,  $\widehat K_{p}$ is independent of $X_{0}$ because the transversal vector field $X_{0}$ is absent in the constant term of equation~(\ref{eq:development-limit}).
\end{proof}

Formula~(\ref{eq:development-limit})  is useful to compute  $\widehat K_{p}$ explicitly, as it contains only derivatives of the submersion $f$; thus, let us enclose it with the following corollary.

\begin{corollary}\label{cor:formula-for-Khat} 
Let $p$  be a characteristic point of $S$. Let $f$ be a local submersion of class $C^{2}$ describing $S$, and let $(X_{1},X_{2})$ be a local oriented orthonormal  frame   of $D$. Then, 
	\begin{equation}\label{eq:Khat-as-hessian}
			\widehat K_{p} = -1 + \frac{\det \hess_{H} f(p)}{[X_{2}, X_{1}]f(p)^{2}}.
	\end{equation}
\end{corollary}
Note that both $\det \hess_{H} f(p)$ and $[X_{2}, X_{1}]f(p)$ calculated at the characteristic point $p$ are invariant with respect to the frame $(X_{1},X_{2})$.
Moreover, we emphasise that their ratio, which appears in \eqref{eq:Khat-as-hessian}, is independent on the choice of $f$.

%% %% %% %% %% %% %% %% %% %% %% %% %% %% %% %% %% %% %
%	SECTION 4 --- LOCAL STUDY
%% %% %% %% %% %% %% %% %% %% %% %% %% %% %% %% %% %% %

\section{Local study near a characteristic point}\label{sec:local-study}
	
	In this section, we prove Proposition~\ref{prop:Khat-and-eigenvalues}, and we discuss the local qualitative behaviour of the characteristic foliation near $\Sigma(S)$ in relation to the metric coefficient $\widehat K$; next, we estimate the length of a semi-leaf converging to a point, proving Proposition~\ref{prop:leavesFiniteLength}.
	
Let us fix a characteristic point $p$ in $\Sigma(S)$, and a characteristic vector field $X$. 
	Since $X(p)=0$, there exists a well-defined linear map $DX(p):T_{p}S\to T_{p}S$. 
Indeed, let $e^{tX}$ be the flow of $X$. 
The pushforward  of the flow gives,  for every $x$ in $S$, a family of linear maps $e^{tX}_{*}:T_{x}S\to T_{e^{tX}(x)}S$.
Since $e^{tX}(p)=p$ for all $t$, then the preceding gives the linear flow $e^{tX}_{*}:T_{p}S\to T_{p}S$, whose infinitesimal generator is the differential $DX(p)$.
	
\begin{definition} \label{defn:non-degenerate}
A characteristic point $p\in \Sigma(S)$  is  \textit{non-degenerate} if, given  a characteristic vector field $X$ of $S$, the differential $DX(p)$ is invertible.
Otherwise, $p$ is called \textit{degenerate}. 
\end{definition}

\begin{remark}\label{rmq:DX-multiple} 
Condition (\ref{eq:non-deneric-cvf}) in the definition of  characteristic vector field ensures that the degeneracy of a characteristic point is independent on the choice of characteristic vector field.
\end{remark}

Since $T_{p}S$ coincides with $D_{p}$ at the characteristic point $p$, we can endow $T_{p}S$ with a metric; thus, $DX(p)$ admits a well-defined determinant and trace. 
Now, let $X$ be the vector field  $X=a_{1}X_{1}+a_{2}X_{2}$, where $(X_{1}, X_{2})$ is an orthonormal oriented frame of $D$ and $a_{i}\in C^{1}(S)$, for $i=1,2$.
Then, in the frame defined by $(X_{1}, X_{2})$ one has
	\begin{equation}\label{eq:DX} 
	DX = \begin{pmatrix}X_{1}a_{1} & X_{2}a_{1} \\ X_{1}a_{2} & X_{2}a_{2} \end{pmatrix},
	\end{equation}
and the formulas for the determinant and the trace are
	\begin{align}
	\det DX &= (X_{1}a_{1})(X_{2}a_{2})-(X_{1}a_{2})(X_{2}a_{1}),\label{eq:det-DX}\\
	 \tr DX &= \dive X =(X_{1}a_{1})+(X_{2}a_{2}). \label{eq:tr-DX}
	\end{align}	

%% %% %% %% %% %% %% %% %% %% %% %% %% %% %% %% %% %% %
	
\subsection{Proof of Proposition~\ref{prop:Khat-and-eigenvalues}} 
Let us fix a characteristic point $p$ in $\Sigma(S)$.
We claim that the right-hand side of~(\ref{eq:Khat-and-DX(p)}) is independent on the choice of the characteristic vector field $X$. 
Indeed, due to Remark~\ref{rmq:proportionality-cvf} any two characteristic vector fields are multiples by nonzero functions, thus, at characteristic point $p$, their differentials are multiples by nonzero scalars; precisely, if $Y=\phi X $, for $\phi$ in $C^{1}(S)$, then  one has $DY(p)=\phi(p)DX(p)$. Thus, the claim follows because both determinant and trace-squared are  homogenous of the degree two. 

Thus, we fix a local submersion $f$ defining $S$ near $p$, and the characteristic vector field $X_{f}=(X_{1}f)X_{2}-(X_{2}f)X_{1}$ defined in \eqref{eq:c-vf}.
Using expression~(\ref{eq:DX}) for the differential of a vector field, we get 
\[
	DX_{f}(p)=
	\begin{pmatrix} 
	-X_{1}X_{2}f(p) & -X_{2}X_{2}f (p)\\
	X_{1}X_{1}f (p)& X_{2}X_{1}f(p)
	\end{pmatrix}.
\]
Thus, using expressions~(\ref{eq:det-DX}) and~(\ref{eq:tr-DX}) for the determinant and the trace, we find that $\det DX_{f}(p) = \det \hess_{H}f (p)$, and $ \tr DX_{f}(p) = [X_{2},X_{1}]f(p)$.
In conclusion, 
\[
	\frac{\det DX_{f}(p) }{\tr DX_{f}(p)^{2}} = \frac{\det \hess_{H}f (p)}{[X_{2},X_{1}]f(p)^{2}},
\]	
which, together with Corollary~\ref{cor:formula-for-Khat}, gives the desired result.

\hfill $\square$
\\

The eigenvalues of  the linearisation $DX(p)$ of a characteristic vector field $X$ can be written as a function of $\widehat K_{p}$ by rearranging equation \eqref{eq:Khat-and-DX(p)}, as in the following corollary.
\begin{corollary} In the hypothesis of Proposition~\ref{prop:Khat-and-eigenvalues}, let $\lambda_{+}(X,p)$ and $\lambda_{-}(X,p)$ be the two eigenvalues of $DX(p)$. Then
\begin{equation}\label{eq:eigenvalues-by-K-hat}
	\lambda_{\pm}(X,p)=\tr DX(p) \bigg(\frac{1}{2}\pm \sqrt{-\frac{3}{4}-\widehat{K}_{p}}\bigg).
\end{equation}
\end{corollary} 

\begin{proof}
	Let us note $\lambda_{\pm}=\lambda_{\pm}(X,p)$, and  $\alpha=\tr DX(p)$. Equation \eqref{eq:Khat-and-DX(p)} reads
	\[
	\widehat K_{p} = -1+\frac{\lambda_{+}\lambda_{-}}{\alpha^{2}}.
	\]
 Using that $\lambda_{+}+\lambda_{-}=\alpha$, equation \eqref{eq:Khat-and-DX(p)} implies that the eigenvalues satisfy the equation $z^{2}-\alpha z+\alpha^{2}(\widehat K_{p} +1)=0$, which implies  \eqref{eq:eigenvalues-by-K-hat}.
\end{proof}

\begin{remark}
It is possible to choose canonically a characteristic vector field with trace 1. Indeed, in the notations used in Remark~\ref{req:local-setting}, let us define the characteristic vector field 
\begin{equation}\label{eq:XS}
	X_{S}=\frac{(X_{1}f)X_{2}-(X_{2}f)X_{1}}{Z f},
\end{equation}
where $Z$ is the Reeb vector field of the contact form $\contact$ of $D$ defined in \eqref{eq:contact-form}, i.e., the unique vector field satisfying $ \contact(Z)=1$ and  $d\contact(Z, \cdot)=0$.
The vector field $X_{S}$ is a characteristic vector field in a neighbourhood of $p$ because it is a nonzero multiple of $X_{f}$ near $\Sigma(S)$, since $Z f(p)=[X_{2},X_{1}]f(p)\neq 0$. 
Using the latter, one can verify that $\dive X_{S}(p)=\tr DX_{S}(p)=1$.

It is worth mentioning that the vector field $X_{S}$ is independent on $f$ and on the frame $(X_{1},X_{2})$, i.e., it depends uniquely on $S$ and $(M,D,g)$. 
Moreover, the norm of $X_{S}$ satisfies $|X_{S}|_{g}^{-1}=|p_{S}|$, where $p_{S}$ is the \textit{degree of transversality} defined in \cite{Lee_2013}; in the case of the Heisenberg group, $p_{S}$ coincides with the \textit{imaginary curvature} introduced in \cite{Arcozzi:2007aa, Arcozzi2008}.
\end{remark}

	Expression (\ref{eq:eigenvalues-by-K-hat}) for the eigenvalues of the linearisation $DX(p)$ implies the following relations between the eigenvalues  and the metric coefficient $\widehat K_{p}$:
	\begin{enumerate}[label=\textit{(\roman*)}]
	\item  \label{it:saddle} $ \widehat{K}_{p}< -1 $ if and only if 
		$  \lambda_{\pm}\in \R^{*}$ with different signs;
	\item \label{it:degenerate} $ \widehat{K}_{p} = -1  $ if and only if 
		$  \lambda_{-}=0$ and $\lambda_{+}\in\R^{*}$;
	\item \label{it:node} $ -1< \widehat{K}_{p} \leq-3/4 $ if and only if 
		$ \lambda_{\pm}\in \R^{*}$ with  same sign;
	\item  \label{it:focus} $-3/4<\widehat{K}_{p} $ if and only if 
		$ \Re(\lambda_{\pm})\neq 0 \neq \Im(\lambda_{\pm})$ and $\lambda_{-}=\overline\lambda_{+} $.
	\end{enumerate}
Notice that the characteristic point $p$ is degenerate if and only if $\widehat{K}_{p} = -1$, which is case \ref{it:degenerate}.
\\

Assume that $p$ is a non-degenerate characteristic point. 
Then, the  linear dynamical system defined  by $DX(p)$ is a saddle,  a node, and a focus respectively in case~\ref{it:saddle},~\ref{it:node} and~\ref{it:focus}.
In these cases  there exists a local $C^{1}$-diffeomorphism near $p$ which sends the flow of $X$ to the flow of $DX(p)$ in $\R^{2}$, i.e., the flows are \textit{$C^{1}$-conjugate}, as proven by Hartman in  \cite{Hartman1960OnLH}.
For this theorem to hold, one needs the characteristic vector field $X$ to be of class $C^{2}$. For this reason,  in the following corollary we assume the surface $S$ to be  of class $C^{3}$.

\begin{corollary} \label{cor:qualitative-char-foli-non-deg} 
	Assume that the surface $S$ is of class $C^{3}$, and let $p$ be a non-degenerate characteristic point in $\Sigma(S)$. 
	Then, $\widehat{K}_{p} \neq -1$, and the characteristic foliation of $S$ in a neighbourhood of $p$ is $C^{1}$-conjugate to
	\begin{itemize}[-] 
		\item   a saddle if and only if  $\widehat{K}_{p}< -1  $;
		\item a node if and only if $-1< \widehat{K}_{p} \leq-3/4$;
		\item a focus if and only if $-3/4<\widehat{K}_{p} $.
	\end{itemize}	
Those chases are depicted, respectively, in the first, third and fourth image  in Figure~\ref{img:isolated-char-point}.
\end{corollary}
 
 \begin{remark}\label{remq:topolog-conju} 
 For surfaces of class  $C^{2}$, i.e., with characteristic vector fields of class $C^{1}$, one can use the Hartman-Grobman theorem, by which one recovers a $C^{0}$-conjugation to the corresponding linearisation. However, under this hypothesis,  a node and a focus become indistinguishable. 
 For the Hartman-Grobman theorem we refer to \cite[Par.~2.8]{perko2012differential}. 
Finally, for a $C^{\infty}$ surface some informations can be found in \cite{GUYSINSKY:2003aa}.
 \end{remark}
 
Next, if $p$ is a degenerate characteristic point, then we are in case~\ref{it:degenerate}. 
Thus, $\widehat K_{p}=-1$, and the differential $DX(p)$ has a zero eigenvalue with multiplicity one.
In this situation, the qualitative behaviour of the characteristic foliation does not depend uniquely on the linearisation, but also on the nonlinear dynamic along a \textit{center manifold}, i.e., an embedded curve $\mathcal{C}\subset S$ with the same regularity as $X$, invariant with respect to the flow of $X$, and tangent to the zero eigenvector of $DX(p)$.
The analogue of Corollary~\ref{cor:qualitative-char-foli-non-deg} is the following.

 \begin{corollary} \label{cor:qualitative-char-foli-deg} 
	Assume that the surface $S$ is of class $C^{2}$, and let $p$ be a degenerate characteristic point in $\Sigma(S)$. 
	Then, $\widehat{K}_{p}=-1$, and the characteristic foliation in a neighbourhood centred at $p$ is $C^{0}$-conjugate at the origin to the orbits of a system of the form
	\begin{equation}\label{eq:system-center-manif}
	\left\{\begin{array}{l}
	\dot u = \phi(u) \\
	\dot v = v
	\end{array}\right.,
	\end{equation}
	for a function $\phi$ with $\phi(0)=\phi'(0)=0$. 
	If $p$ is isolated, then the characteristic foliation described in \eqref{eq:system-center-manif}  at the origin is either a saddle, a saddle-node, or a node; those cases are depicted, respectively, in the first, second, and third image in Figure~\ref{img:isolated-char-point}.
\end{corollary}

\begin{figure}[]
\begin{center}
\includegraphics[width=0.95\linewidth]{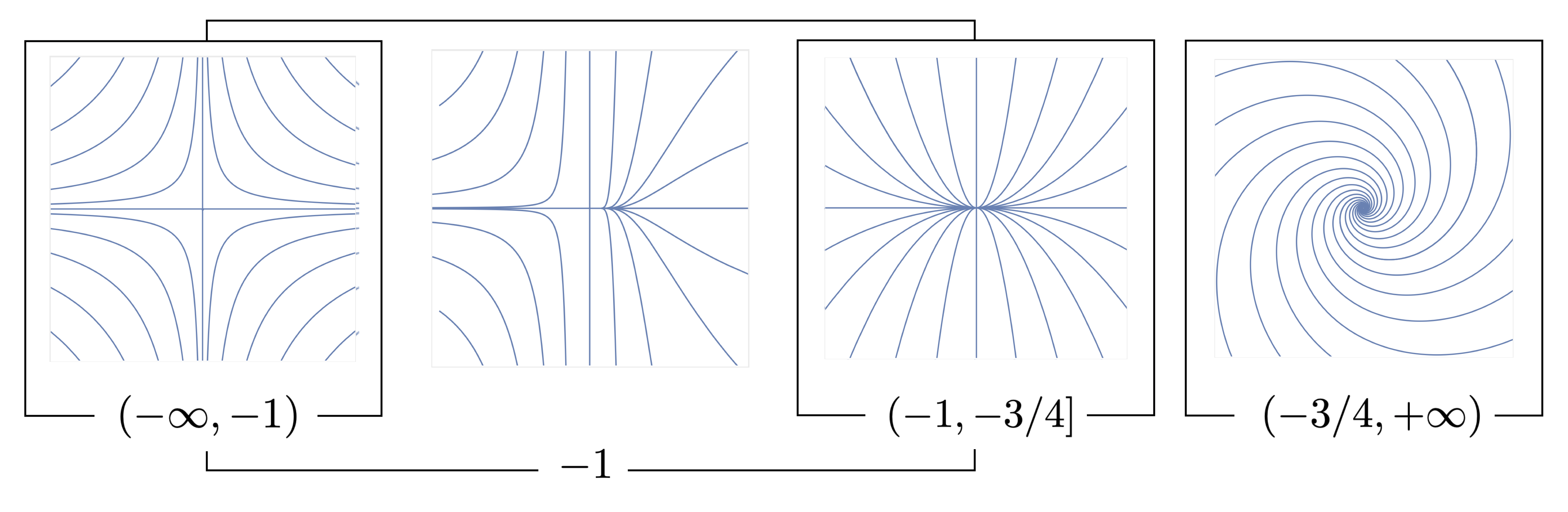}
\caption{\small The qualitative picture for the characteristic foliation at an \textit{isolated} characteristic point, with the corresponding values for~$\widehat K$. From left to right, we recognise a  saddle, a  saddle-node, a  node, and a  focus.} 
\label{img:isolated-char-point}
\end{center}
\end{figure}

	The proof of Corollary~\ref{cor:qualitative-char-foli-deg} follows from considerations on the center manifold of the dynamical system defined by $X$, which we recall in Appendix~\ref{append:center-manif}.

\begin{remark} 
A node and a focus are not distinguishable by a conjugation $C^{0}$. 
However, the center manifold of the characteristic point $p$  is an embedded curve of class $C^{1}$, thus it does not spiral around $p$. 
Therefore, the existence of a center manifold gives further properties then what is expressed in Corollary~\ref{cor:qualitative-char-foli-deg}. 
\end{remark}

To justify the last sentence of Corollary~\ref{cor:qualitative-char-foli-deg} let us get a sense of the qualitative properties of a system as (\ref{eq:system-center-manif}).
The line $\{v=0\}$, parametrised by $u$, is a center manifold of (\ref{eq:system-center-manif}), and the function $\phi$ determines the dynamic of (\ref{eq:system-center-manif}); this illustrates the fact that the nonlinear terms on a center manifold determine the dynamic near a degenerate characteristic point. 
\\

The equilibria of (\ref{eq:system-center-manif}) occur only in $\{v=0\}$, i.e., on a center manifold, and a point $(u,0)$ is an equilibrium if and only if $\phi(u)=0$.
Thus, if the characteristic point $p$ is isolated, then  $u_{0}=0$ is an isolated zero of $\phi$. 
In such case, let us note $\phi^{+}=\phi|_{u>0}$ and $\phi^{-}=\phi|_{u<0}$, and without loss of generality let us suppose that the signs of $\phi^{+}$ and $\phi^{-}$ are constant.
\begin{itemize}[-]
\item If $\phi^{+}>0$ and $\phi^{-}<0$, then the origin is a topological node. 
\item If $\phi^{+}<0$ and $\phi^{-}>0$, then the origin is a a topological saddle. 
\item If $\phi^{+}$ and $\phi^{-}$ have the same sign, then the two half spaces $\{u>0\}$ and $\{u<0\}$ have two different behaviours: one is a node, and the other one is a saddle. This gives the characteristic foliation called saddle-node.
\end{itemize} 

\begin{remark}\label{req:local-picture-for-isolated}
	For an isolated characteristic point, combining Corollary~\ref{cor:qualitative-char-foli-non-deg} and Corollary~\ref{cor:qualitative-char-foli-deg},  we obtain the four characteristic foliations depicted in Figure~\ref{img:isolated-char-point}. 
\end{remark}

%% %% %% %% %% %% %% %% %% %% %% %% %% %% %% %% %% %% %

\subsection{Proof of Proposition~\ref{prop:leavesFiniteLength}} 
In this section we prove the finiteness of the sub-Riemannian length of a semi-leaf converging to a point.
Since we are interested in a local property, it is not restrictive to assume the existence of a global characteristic vector field $X$ of $S$.

	Let $\ell$ be a one-dimensional leaf of the characteristic foliation of $S$, and $x\in \ell$ such that $e^{tX}(x)\to p$ as $t\to +\infty$. 
The limit point $p$ has to be an equilibrium of $X$, i.e., $X(p)=0$, hence $p$ is a characteristic point of $S$. 
Let $U$ be a small open  neighbourhood of $p$ in $S$ for which we have a coordinate chart $	\Phi : U\to B\subset \R^{2}$ with $\Phi(p)=0$, where $B$ is the open unit ball. Let $y$ be the  point of last intersection between $\ell^{+}_{X}(x)$ and the boundary $\partial U$.
Since  $L_{sR}(\ell^{+}_{X}(x))=L_{sR}(\ell|_{[x,y]})+L_{sR}(\ell^{+}_{X}(y))$ and $L_{sR}(\ell|_{[x,y]})$ is finite, it suffices to show that $L_{sR}(\ell^{+}_{X}(y))$ is finite.
We claim that there exists a constant $C>0$ such that 
\begin{equation}\label{eq:equivalence-metric}
	\frac{1}{C}|V|_{\R^{2}} \leq |V|_{g} \leq C |V|_{\R^{2}} \qquad \forall ~ V\in D\cap TS|_{U},
\end{equation}
where we have dropped  $\Phi_{*}$ in the notation. 
Indeed, let $\tilde g$ be any Riemannian extension of $g$ on the surface $S$ (for example $\tilde g=g^{X_{0}}|_{S}$).  Since $\tilde g$ is an extension, one has  $|v|_{g}=|v|_{\tilde g}$ for all $v$ in $D\cap TS$.
Equivalence \eqref{eq:equivalence-metric} follows from the local equivalence of $\tilde g$ with the pullback  by $\Phi$ of the Euclidean metric of $\R^{2}$.
Now, inequality  \eqref{eq:equivalence-metric} implies that
\begin{equation}\label{eq:sR-length-of-leaf}
	L_{sR}(\ell_{X}^{+}(y)) =\int_{0}^{+\infty}|X(e^{tX}(y))|_{g}~dt\leq C \int_{0}^{+\infty}|X(e^{tX}(y))|_{\R^{2}}~dt.
\end{equation}
At this point the proof of the finiteness of the sub-Riemannian length of $\ell_{X}^{+}(y)$  differs depending on whether $p$ is a non-degenerate or a degenerate characteristic point. \\

	First, assume that $p$ is a non-degenerate characteristic point.
Since $p$ is non-degenerate, then the set of point $w$ with $e^{tX}(w)\to p$ for $t\to+\infty$ form a manifold, called the \textit{stable manifold} at $p$ for the dynamical system defined by $X$.
In our case, since $e^{tX}(y)\to p$ for $t\to+\infty$, the semi-leaf $\ell_{X}^{+}(y)$ is contained in the stable manifold at $p$.
Moreover, the stable manifold convergence property,  precisely stated in~\cite[Par.~2.8]{perko2012differential}, shows that each trajectory inside the stable manifold converges to $p$ sub-exponentially in $t$. Precisely, if $\alpha$ satisfies $|\Re(\lambda_{\pm}(p,X))|>\alpha$, then there exists constants $C,t_{0}>0$ such that 
\begin{equation}\label{eq:approx-property}
	|e^{tX}(y)-p|_{\R^{2}}\leq Ce^{-\alpha t} \qquad \forall ~ t>t_{0}.
\end{equation}
Since $X(p)=0$, for all $t>0$  one has
\[
	\big|X(e^{tX}(y))\big|_{\R^{2}} = \big|X(e^{tX}(y)) - X(p) \big|_{\R^{2}} \leq \sup_{B}||DX(x)|| ~ |e^{tX}(y)-p|_{\R^{2}}.
\]
Due to the inequality~(\ref{eq:sR-length-of-leaf}) and~(\ref{eq:approx-property}), this shows that $L_{sR}(\ell_{X}^{+}(y))$ is finite.\\

	Next, assume that $p$ is a degenerate characteristic point. 
As we said in the introduction of Corollary~\ref{cor:qualitative-char-foli-deg}, there exists a center manifold $\mathcal{C}$ at $p$ for the dynamical system defined by $X$.
The asymptotic approximation property of the center  manifold, recalled in Proposition~\ref{approxCM}, shows that if a trajectory converges to $p$, then it approximates any center manifold exponentially fast. 
Precisely,  since $e^{tX}(y)\to p$, then there exist  constants $C,\alpha,t_{0}>0$ and a trajectory $e^{tX}(z)$ contained in $\mathcal{C}$, such that
\begin{equation} \label{eq:asyn-approx-property}
	|e^{tX}(y)-e^{tX}(z)|_{\R^{2}}\leq C e^{-\alpha t} \qquad \forall \ t \geq t_{0}.
\end{equation}
The triangle inequality implies that
	\begin{equation}\label{eq:333}
	|X( e^{tX}(y))|_{\R^{2}} \leq |X(e^{tX}(y))-X(e^{tX}(z))|_{\R^{2}}+  |X(e^{tX}(z))|_{\R^{2}}.
	\end{equation}
Due to inequality~\eqref{eq:sR-length-of-leaf}, to prove that $L_{sR}(\ell_{X}^{+}(y))$ is finite, it suffices to show that the two terms on the right-hand side of \eqref{eq:333} are integrable for $t\geq 0$. Thanks to~(\ref{eq:asyn-approx-property}) and

\[
	 |X(e^{tX}(y))-X(e^{tX}(z))|_{\R^{2}}\leq \sup_{B}||DX|| ~|e^{tX}(y)-e^{tX}(z)|_{\R^{2}},
\]
then the first term in \eqref{eq:333} is integrable.  
Next,  because $e^{tX}(z)$ is a regular parametrisation of a bounded interval inside a $C^{1}$ embedded curve (the center manifold $\mathcal{C}$), then its derivative $|X(e^{tX}(z))|_{\R^{2}}$ is integrable.
 \hfill $\square$
 
	\begin{remark}\label{rmq:finite-length-implies-convergence}
Let $X$ be a characteristic vector field of a compact surface $S$. If the $\omega$-limit set with respect to $X$ of a non-periodic leaf $\ell$ contains more then one point, then $L_{sR}(\ell_{X}^{+})=+\infty$.
Therefore, if a leaf $\ell$ does not converge to a point in any of its extremities, then the points in $\ell$ have infinite distance from the points in $S\mysetminus \ell$.
\end{remark}

In particular, if the characteristic set of a surface $S$ is empty, then the induced distance $d_{S}$ is not finite. 
For a discussion on non-characteristic domains we refer to \cite[Ch.~3]{CGN06}.

%% %% %% %% %% %% %% %% %% %% %% %% %% %% %% %% %% %% %
%	SECTION 5 --- PROOF TOPOLOGICAL SKELETON
%% %% %% %% %% %% %% %% %% %% %% %% %% %% %% %% %% %% %

\section{Global study of the characteristic foliation} \label{sect:topolog-skeleton}
The main goal of this section is to identify  a sufficient condition for the induced distance~$d_{S}$ to be finite.
As explained in the introduction, this is done by excluding the existence of certain leaves in the characteristic foliation of $S$, as in Proposition~\ref{prop:dS-finite}. 
In this section  we assume the existence of a global characteristic vector field $X$ of $S$. 

The leaves the characteristic foliation of $S$ are precisely the orbits of the dynamical system defined by $X$, therefore we are going to call them \textit{trajectories}, stressing that they are parametrised by the flow of $X$. 
Moreover, the vector field $X$ enables us to use the notions of $\omega$-limit set and $\alpha$-limit set of a point $y$ in $S$, which are, respectively,
\[
	\omega(y,X) 
		= \Big\{ q \in S ~\Big|~ \exists ~ t_{n} \to +\infty \mbox{ such that } ~ e^{t_{n}X}(y)\to q \Big\}, \qquad
	 \alpha(y,X)=\omega(y,-X). 
\]
 The points $y$ in a leaf $\ell$ have the same limit sets, thus one  can define $\omega(\ell,X)$ and $\alpha(\ell,X)$.

\begin{proposition}\label{prop:dS-finite}  
Let $S$ be a compact, connected surface $C^{2}$ embedded in a contact sub-Riemannian  structure.
Assume that $S$  has isolated characteristic points, and that the characteristic foliation of $S$ is described by a global characteristic vector field of $S$ which does not contain any of the following trajectories:
\begin{itemize}[-]
\item nontrivial recurrent trajectories, 
\item periodic trajectories,
\item sided contours.
\end{itemize}
Then, $d_{S}$ is finite.
\end{proposition}

Let us give a formal definition of these objects.
A \textit{periodic} trajectory is a leaf of the characteristic foliation homeomorphic to a circle. 
A periodic trajectory has infinite distance from its complementary, hence it is necessary to exclude its presence for $d_{S}$ to be finite.

Next, a leaf $\ell$ is \textit{recurrent} if $\ell\subset\omega(\ell,X)$ and $\ell\subset\alpha(\ell,X)$. 
A \textit{nontrivial} recurrent  trajectory is a recurrent trajectory which is not an equilibrium nor a periodic trajectory. 
Because the $\omega$-limit   and the $\alpha$-limit set of a nontrivial recurrent trajectory contains more then one point,  then, due to Remark~\ref{rmq:finite-length-implies-convergence}, those trajectories have infinite distance from their complementary.

Lastly, a \textit{sided contour} is either a left-sided or right-sided contour. A right-sided contour (resp.~left-sided) is a family of points $p_{1}, \dots, p_{s}$ in $\Sigma(S)$ and  trajectories $\ell_{1}, \dots, \ell_{s}$ such that:
\begin{itemize}[-]
\item for all $j=1,\dots,s$, we have $\omega(\ell_{j},X)=p_{j}=\alpha(\ell_{j+1},X)$ (where $\ell_{s+1}=\ell_{1}$);
\item for every $j=1,\dots,s$, there exists a neighbourhood $U_{j}$ of $p_{j}$ such that $U_{j}$ is a right-sided  hyperbolic sector (resp.~left-sided) for  $p_{j}$ with respect to $\ell_{j}$ and $\ell_{j+1}$.
\end{itemize}
Let us give a precise definition of a hyperbolic sector. Note that, given a non-characteristic point $x\in S$, and a curve $T$ going through $x$ and transversal to the flow of $X$, the orientation defined by $X$ defines the right-hand and the left-hand connected component of $T\mysetminus\{x\}$, denoted $T^{r}$ and $T^{l}$ respectively.

\begin{definition}\label{defn:hyperbolic-sector}
Let $p$ be a characteristic point, and $\ell_{1}$ and $\ell_{2}$ be two trajectories such that $\omega(\ell_{1},X)=p=\alpha(\ell_{2},X)$.
A neighbourhood $U$ of $p$ homeomorphic to a disk is a \textit{right-sided hyperbolic sector}  (resp.~left-sided) with respect to $\ell_{1}$ and $\ell_{2}$  if, for every point   $x_{i}\in \ell_{i} \cap U$, for $i=1,2$, there exists a curve $T_{i}$ going through $x_{i}$ and transversal to the flow of $X$ such that:
\begin{itemize}[-]
\item for every point $y\in T_{1}^{r}$  (resp.~$T_{1}^{l}$) the positive semi-trajectory $\ell_{X}^{+}(y)$ starting from $y$ intersects $T_{2}^{r}$  (resp.~$T_{2}^{l}$)  before leaving~$U$;
\item the point of first intersection of $\ell^{+}_{X}(y)$ and  $T_{2}^{r}$  (resp.~$T_{2}^{l}$) converges to $x_{2}$, for $y\to x_{1}$.
\end{itemize}
\end{definition}
\begin{figure}[h]
 	\centering
	 \includegraphics[width=0.45\linewidth]{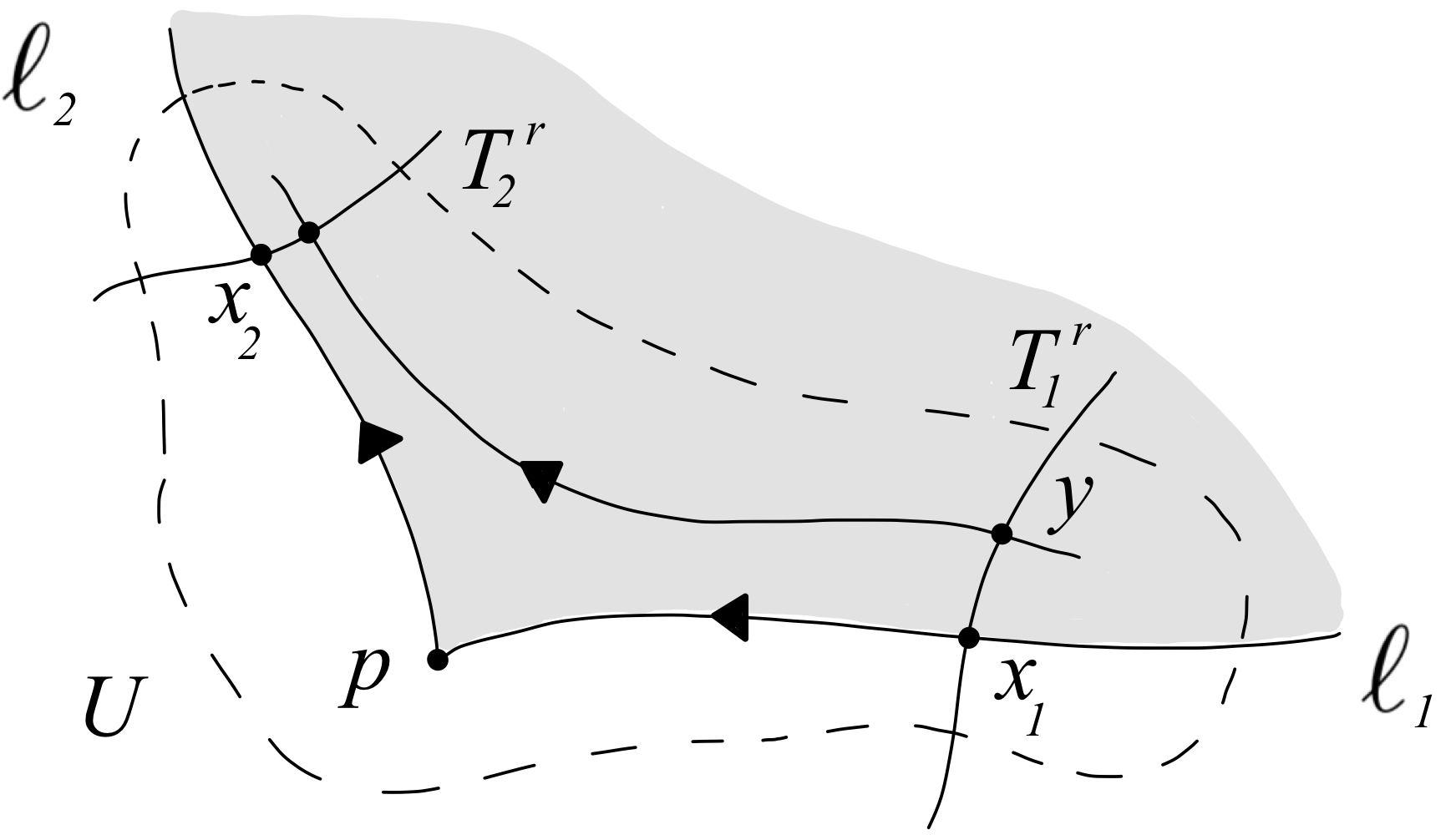}
	\caption{\small The illustration of a right-sided hyperbolic sector}
	 \label{fig:hyperbolic-sector}
\end{figure}
Note that a right-sided hyperbolic sector for $X$ is a left-sided hyperbolic sector for $-X$.
An illustration of hyperbolic sector can be found in Figure~\ref{fig:hyperbolic-sector}, an example of sided contours can be found in  Figure~\ref{fig:sided-contours}, and for the general theory we refer to  \cite[Par.~2.3.5]{ABZintrod}. 

%% %% %% %% %% %% %% %% %% %% %% %% %% %% %% %% %% %% %
	
	\subsection{Topological structure of the characteristic foliation}
	
	Now, assume that $S$ does not contain any nontrivial recurrent trajectories. 
To prove Proposition~\ref{prop:dS-finite}  we are going to use the topological structure of a flow. 
We resume here the relevant theory, following the exposition in \cite[Par.~3.4]{ABZintrod}.
\\

	The {\it singular trajectories} of the characteristic foliation of $S$ are precisely the  following: 
\begin{enumerate} [-]
	\item  characteristic points;
	\item  separatrices of characteristic points (see \cite[Par.~2.3.3]{ABZintrod});
	\item  isolated periodic trajectories;
	\item  periodic trajectories which contain in every neighbourhood both  periodic and non-periodic trajectories.
\end{enumerate}
The union of the singular trajectories is noted $ST(S)$, and it is closed.
The open connected components of $S\mysetminus ST(S)$ are called \emph{cells}. 
The leaves of the characteristic foliation of $S$ contained in the same cell have the same behaviour, as shown in the following proposition.

	\begin{proposition}[{\cite[Par.~3.4.3]{ABZintrod}}] \label{prop:property-cell}
Assume that the flow of $X$ has a finite number of singular trajectories. Let $R$ be a cell filled by non-periodic trajectories; then: \begin{enumerate}[(i)]
\item \label{ref:homeo} $R$ is homeomorphic to a disk, or to an annulus;
\item  the trajectories contained in $R$ have all the same $\omega$-limit and $\alpha$-limit sets;
\item the limit sets of any trajectory in $R$ belongs to $\partial R$;
\item  each connected component of $\partial R$ contains points of the $\omega$-limit or $\alpha$-limit sets.
\end{enumerate}
\end{proposition}

Using this proposition, we show the following lemma.

\begin{lemma} \label{LemmaPropertyBondary} 
Let $S$ be surface satisfying the hypothesis of Proposition~\ref{prop:dS-finite} . 
Then, for every cell $R$ of the characteristic foliation of $S$, we have that 
\begin{equation*} \label{distanzaCellBondary}
	d_{S}(x,y)<+\infty \qquad \forall x ,y\in R \cup \partial R.
\end{equation*}
\end{lemma}

\begin{proof} 
Since the surface $S$ is compact and the characteristic points in $\Sigma(S)$ are isolated,  there is a finite number of characteristic points. 
Moreover, there are no periodic trajectories.
This implies that there is a finite number of singular trajectories, hence we can apply Proposition~\ref{prop:property-cell}.

Let $R$ be a cell of the characteristic foliation of $S$, and let $\Gamma$ be one of the connected components of the boundary $\partial R$ (of which there are either one or two, due to Proposition~\ref{prop:property-cell}). 
The curve $\Gamma$ is the union of characteristic points and separatrices.
If all characteristic points have a hyperbolic sector towards $R$ (right-sided or left-sided), then $\Gamma$ would be a sided contour, which is excluded. 
Therefore, there exists a characteristic point $p\in \Gamma$ without a hyperbolic sector towards $R$.
As shown in \cite[Par.~8.18]{Andronov-1973}, around an isolated equilibrium there are only the three kinds of sectors depicted in Figure~\ref{img:local-sectors}.
Since there is no elliptic sector due to Remark~\ref{req:local-picture-for-isolated}, the point $p$ has a parabolic sector towards $R$.

\begin{figure}[]
\begin{center}
\includegraphics[width=0.45\linewidth]{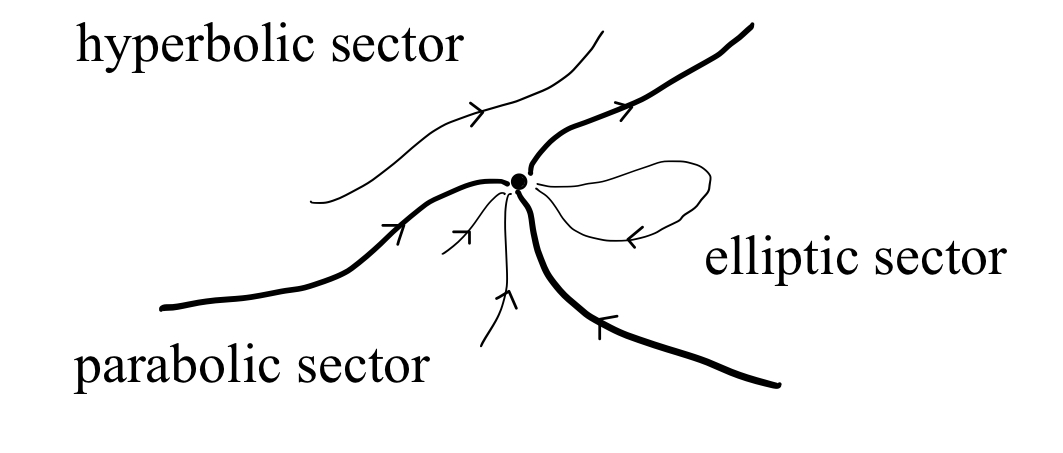}
\caption{\small The  sectors of an isolated equilibrium of a dynamical system.}
\label{img:local-sectors}
\end{center}
\end{figure}

Due to Proposition~\ref{prop:property-cell},  the point $p$ is the $\omega$-limit or the $\alpha$-limit of every trajectory in $R$.
Then, for every point $x\in R$, there exists a semi-leaf  $\ell^{+}_{ X}(x)$ or  $\ell^{+}_{-X}(x)$ starting from $x$ and converging to $p$. 
Due to Proposition~\ref{prop:leavesFiniteLength}, this semi-leaf has finite sub-Riemannian length, hence  $d_{S}(x,p)$ is finite. 

Next, for every point $y \in \Gamma$, note that  \[d_{S}(x,y)\leq d_{S}(x,p)+d_{S}(p, y).\]
We have already proven that $d_{S}(x,p)$ is finite, and the same holds for $d_{S}(p, y)$. Indeed, one can find a horizontal curve of finite length connecting $p$ and $y$  using a concatenation of the separatrices contained in $\Gamma$. 
	
If the boundary of $R$ has a second connected component, then the above argument holds also for the other connected component because it suffices to  repeat the above argument for it. 
Thus, we have shown that
\begin{equation*}
	d_{S}(x,y)<+\infty \qquad \forall ~ x  \in R, ~  y\in \partial R,
\end{equation*}
which implies the statement of the lemma.
\end{proof}

\begin{figure}[]
\begin{center}
\includegraphics[width=0.42\linewidth]{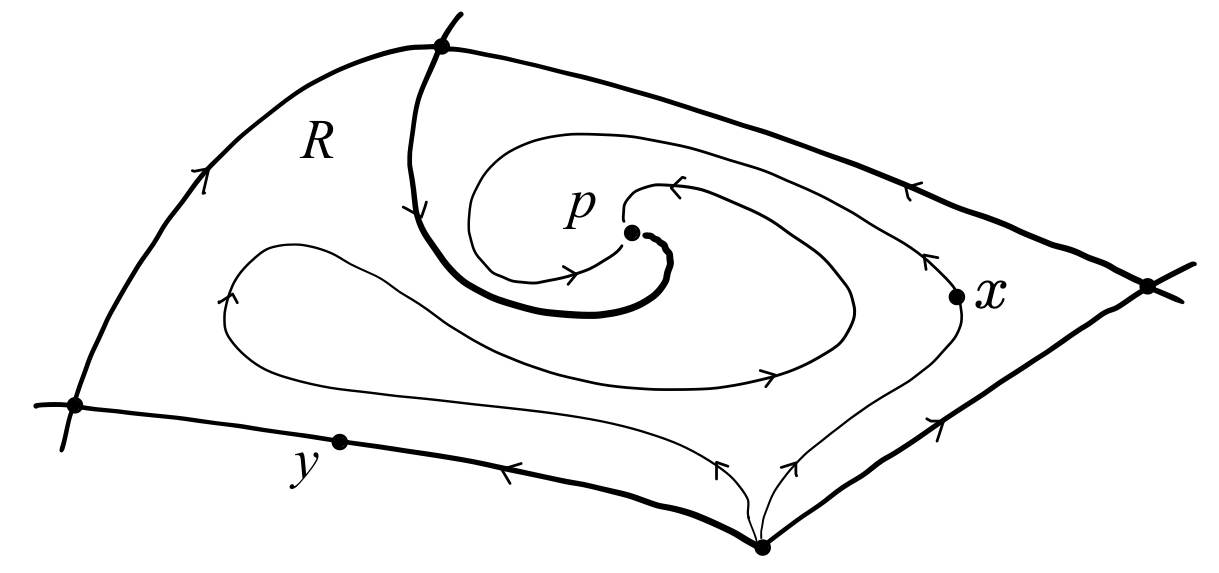} 
 \caption{\label{DimBoundary} \small How to connect the points of a cell with the points in the boundary. }
 \end{center}
 \end{figure}

\begin{lemma} \label{finalConcusion} 
Let $S$ be surface satisfying the hypothesis of Proposition~\ref{prop:dS-finite}.  
Then, for every  $x$ in $S$, there exists an open neighbourhood $U$ of $x$ such that, for all $y$ in $U$,
\[ 
	d_{S}(x,y)< +\infty.
\]
\end{lemma}

\begin{proof} 
Let $x$ be a point of $S$. 
If $x$ does not belong to the union of the singular trajectories, then it is in the interior of a cell $R$. 
Thus, due to Lemma~\ref{LemmaPropertyBondary}, one can choose  $U=R$.
Otherwise, the point $x$ belongs to a separatrix, or it is a characteristic point of $S$. 

Assume that $x$ belongs to a separatrix $\Gamma$. 
Then, there exists a neighbourhood $U$ of $x$ which is divided by $\Gamma$ in two connected components. 
Those two connected components are contained in some cell $R_{1}$ and $R_{2}$, which contain~$\Gamma$ in their boundary. 
For every $y\in U$, then either $y \in R_{i}$,  for $i=1,2$, or  $y \in \Gamma$. 
If  $y \in R_{i}$, then it suffices to apply Lemma~\ref{LemmaPropertyBondary}. Otherwise, if $y \in \Gamma$, the separatrix  $\Gamma$ itself connects  $x$ and $y$.
	
Finally, assume that $x$ is a characteristic point. Due to Corollary~\ref{cor:qualitative-char-foli-non-deg}, Remark~\ref{remq:topolog-conju}, and Corollary~\ref{cor:qualitative-char-foli-deg},  there exists a neighbourhood $U$ of $x$ in which the characteristic foliation of $S$ is topologically conjugate to a saddle, a node or a saddle-node. 
Thus, one can repeat the same argument as before: for every $y\in U$, if $y$ belongs to a cell then one applies Lemma~\ref{LemmaPropertyBondary}; otherwise, if $y$ belongs to a separatrix one can connect $x$ and $y$ directly.
\end{proof}

The proof of Proposition~\ref{prop:dS-finite}  is an immediate corollary of Lemma~\ref{finalConcusion}.

\begin{proof}[Proof of Proposition~\ref{prop:dS-finite} ] The property of having finite distance is an equivalence relation on the points of $S$. Because of Lemma~\ref{finalConcusion}, the equivalence classes are open.
Thus, because $S$ is connected, there is only one class.
\end{proof}

\begin{figure}[h]
 	\centering
	 \includegraphics[width=0.35\linewidth]{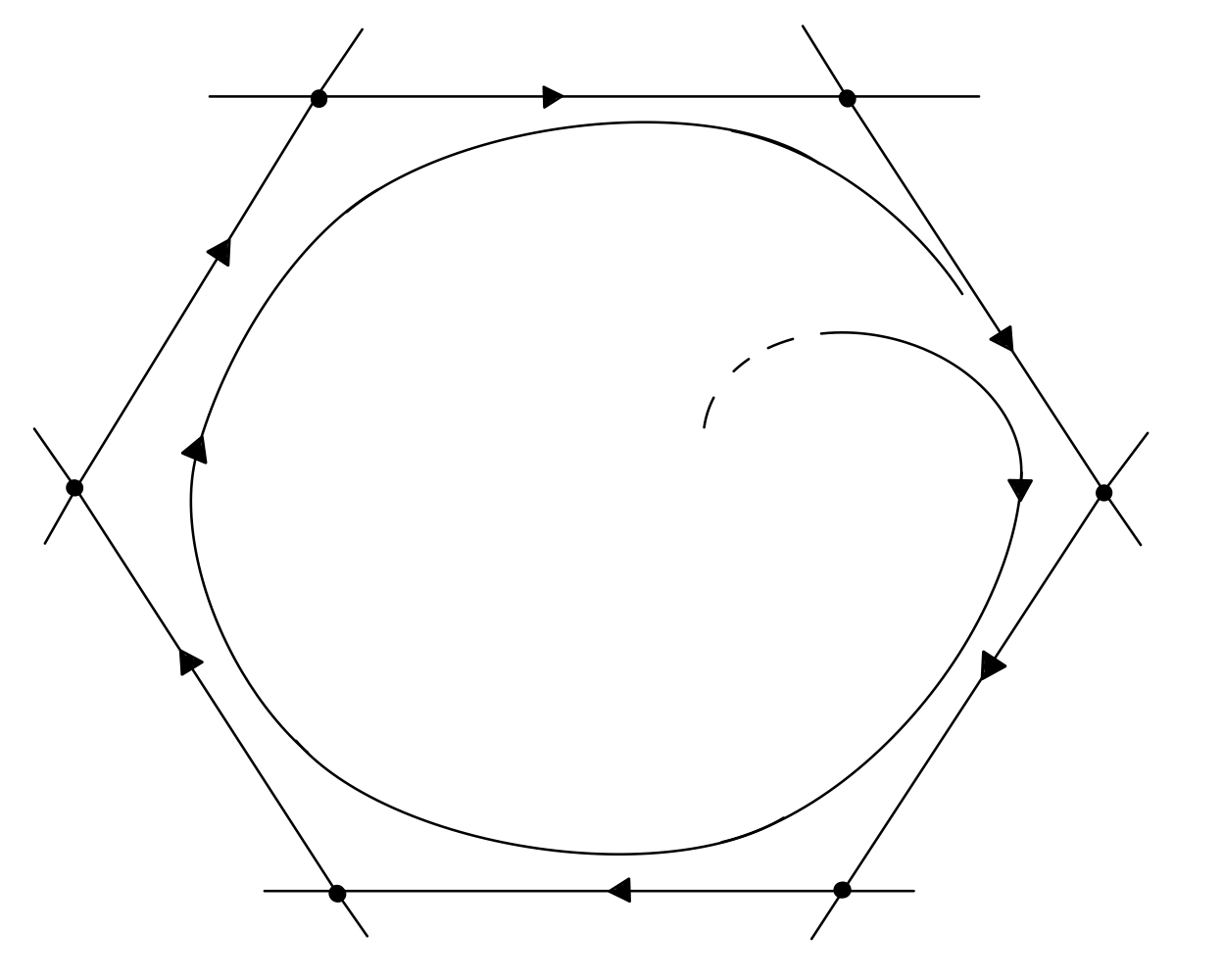}
	\caption{\small An embedded polygon which bounds a right-sided contour}
	 \label{fig:sided-contours}
\end{figure}

%% %% %% %% %% %% %% %% %% %% %% %% %% %% %% %% %% %% %
%	SECTION 6 --- SPHERES
%% %% %% %% %% %% %% %% %% %% %% %% %% %% %% %% %% %% %

\section{Spheres in a tight contact distribution}\label{sect:sphere}

In this section we prove Theorem~\ref{thm:FiniteMetricSph}, i.e.,  in a  tight coorientable contact distribution the topological spheres have finite induced distance. 
This  is  done by showing that the hypothesis of Proposition~\ref{prop:dS-finite}  are satisfied in this setting. \\

An overtwisted disk, precisely defined in Definition~\ref{defn:overtwisted}, is en embedding of a disk with horizontal boundary such that the distribution does not twist along the boundary. 
A contact distribution is called \textit{overtwisted} if it admits a overtwisted disk, and it is called  \textit{tight} if it is non-overtwisted. 

\begin{remark}\label{rmq:no-periodic-in-tight}
Note that if the boundary of a disk is a periodic trajectory of its characteristic foliation, then the disk is overtwisted. 
Indeed, since a periodic trajectory does not contain characteristic points, then the plane distribution never coincides with the tangent space of the disk, thus the distribution can't perform any twists.
 \end{remark}

\begin{lemma} \label{lem:no-periodic-trajectories}
Let $(M,D)$ be a tight contact 3-manifold, and $S$ an embedded surface with the topology of a sphere. 
Then, the characteristic foliation of $S$ does not contain periodic trajectories.
\end{lemma}
\begin{proof}
Assume that the characteristic foliation of $S$ has a periodic trajectory $\ell$. 
Then, because $\ell$ does not have self-intersections, %a generalisation of the Jordan curve theorem implies that 
the leaf $\ell$ divides $S$ in two topological half-spheres $\Delta_{1}$ and $\Delta_{2}$.
The disks $\Delta_{i}$, for $i=1,2$, are overtwisted, which contradicts the hypothesis that the distribution is tight because Remark~\ref{rmq:no-periodic-in-tight}.
\end{proof}

Now, let us discuss the sided contours.

\begin{lemma} \label{lem:no-sided-contours}
Let $(M, D)$ be a tight contact 3-manifold, and $S \subset M$ an embedded surface with the topology of a sphere. 
Then the characteristic foliation of $S$ does not contain sided contours.
\end{lemma}

\begin{proof}
Assume that the characteristic foliation presents a sided contour $\Gamma$. 
Its complementary $S\mysetminus \Gamma$ has two connected components, which are topologically half-spheres. Let us call $\Delta$ the component on the same side of $\Gamma$, i.e., if $\Gamma$ is right-sided (resp.~left-sided) then $\Delta$ is on the right (resp.~left). 
For instance, if $\Gamma$ is right-sided, then the  characteristic foliation of $\Delta$ looks like that of the polygon in Figure~\ref{fig:sided-contours}. 

Let $p$ be one of the vertices of $\Delta$, let  $\ell_{1}$ and $\ell_{2}$ be the separatrices adjacent to~$p$, and let  $U$ be a neighbourhood  of $p$ such that we are in the condition of Definition~\ref{defn:hyperbolic-sector}. 
Let us fix two points $x_{i}\in \ell_{i}\cap U$, for $i=1,2$.  Due to  the definition of hyperbolic sector,  in a neighbourhood of $x_{1}$ the leaves pass arbitrarily close to $x_{2}$. 

We are going to give the idea  of how to perturb the surface  near $x_{1}$ and $x_{2}$ so that the separatrices $\ell_{1}$ and $\ell_{2}$ are diverted to the same nearby leaf, therefore bypassing $p$.
In other words, via a $C^{\infty}$-small perturbation of $S$ supported in a neighbourhood of $x_{1}$ and $x_{2}$, we obtain a sphere which contains a sided contour with one less vertex, see Figure~\ref{img:perturbation}. 
By  repeating such perturbation for every vertex, one obtains a new surface with a periodic trajectory in its  characteristic  foliation, which is excluded due to Lemma~\ref{lem:no-periodic-trajectories}. 
\begin{figure}[]
\begin{center}
\includegraphics[width=0.35\linewidth]{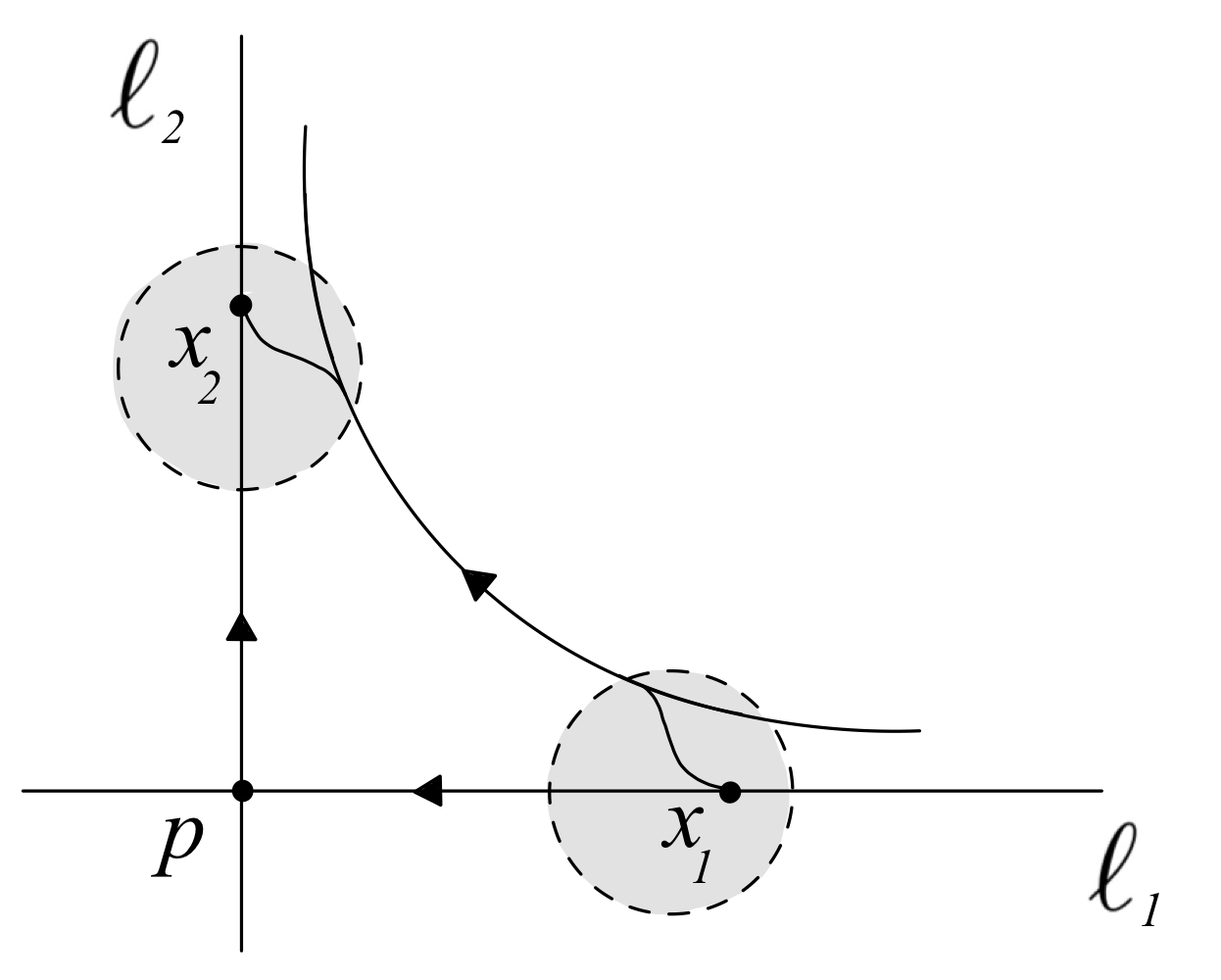}
\caption{\label{img:perturbation}\small The characteristic foliation of the perturbed surface.}
\end{center}
\end{figure}

Consider  the Heisenberg distribution $(\R^{3}, \ker(dz+\frac{1}{2}(y dx -x dy ))$. Let $\mathcal{P}$ be the vertical plane $\mathcal{P}=\{x=0\}$, and $q$ a point in $\mathcal{P}$ contained in the $y$-axis. 
As one can see in Example~\ref{exmp:planes}, the characteristic foliation of $\mathcal{P}$ is made up of parallel horizontal lines. 

Locally, it is possible to rectify the surface $S$ into the plane $\mathcal{P}$ using a contactomorphism of the  respective ambient spaces, as explained in the following lines.
Due to the rectification theorem of dynamical systems, the characteristic foliation of $S$ in a neighbourhood of $x_{1}$ is diffeomorphic to that of a neighbourhood of $q$ in $\mathcal{P}$.
A generalisation of a theorem of Giroux \cite[Thm.~2.5.23]{geiges2008introduction} implies that the $C^{1}$-conjugation between the characteristic foliations of the two surfaces can be extended, in a smaller neighbourhood, to  a contactomorphism. 
Precisely, there exists a contactomorphism $\psi$ from a neighbourhood $V\subset M$ of $x_{1}$ to a neighbourhood of $q$ in $\R^{3}$, with $\psi(S)\subset\mathcal{P}$.

For what it has been said above, the image  of $\ell_{1}$ by $\psi$ is contained in the $y$-axis. 
By creating a small bump 
 in $\mathcal{P}$ after the point $q$, we will be able to divert the leaf going through $q$ to any other parallel line.
Precisely, for any curve $ \gamma(t)=(x(t),y(t))$, defining
 \[ 
 	z(t) =  \frac{1}{2}\int_{t_{1}}^{t} x(s)y'(s) - y(s)x'(s) ds \qquad \forall ~ t\in [t_{1},t_{2}],
\]  
we obtain a horizontal curve $(x(t),y(t),z(t))$.
Now, let $\gamma$ be a smooth curve which joins smoothly to the $y$-axis at its end points $\gamma(t_{1})=q$ and $\gamma(t_{2})$, and let $\Omega$ be the set between $\gamma$ and the $y$-axis.
One can verify that $z(t_{2})= \area\big(\Omega\big)$, where the area is a signed area.
By choosing an appropriate curve $\gamma$, we can connect the $y$-axis from $q$ to any other parallel line in $\mathcal{P}$ via a horizontal curve (Figure~\ref{fig:hyperbolic-perturbation}).
Next, by creating a small bump in~$\mathcal{P}$  in order to include this horizontal curve one has successfully diverted the leaf.
This procedure can be done $C^{\infty}$-small, provided one wants to connect to parallel lines sufficiently close to the $y$-axis. 
Thus, one can make sure that no new characteristic points are created. 
 Finally, this perturbation has to be transposed to a perturbation of $S$ using $\psi$.
 \begin{figure}[] \begin{center}
	\includegraphics[width=0.35\linewidth]{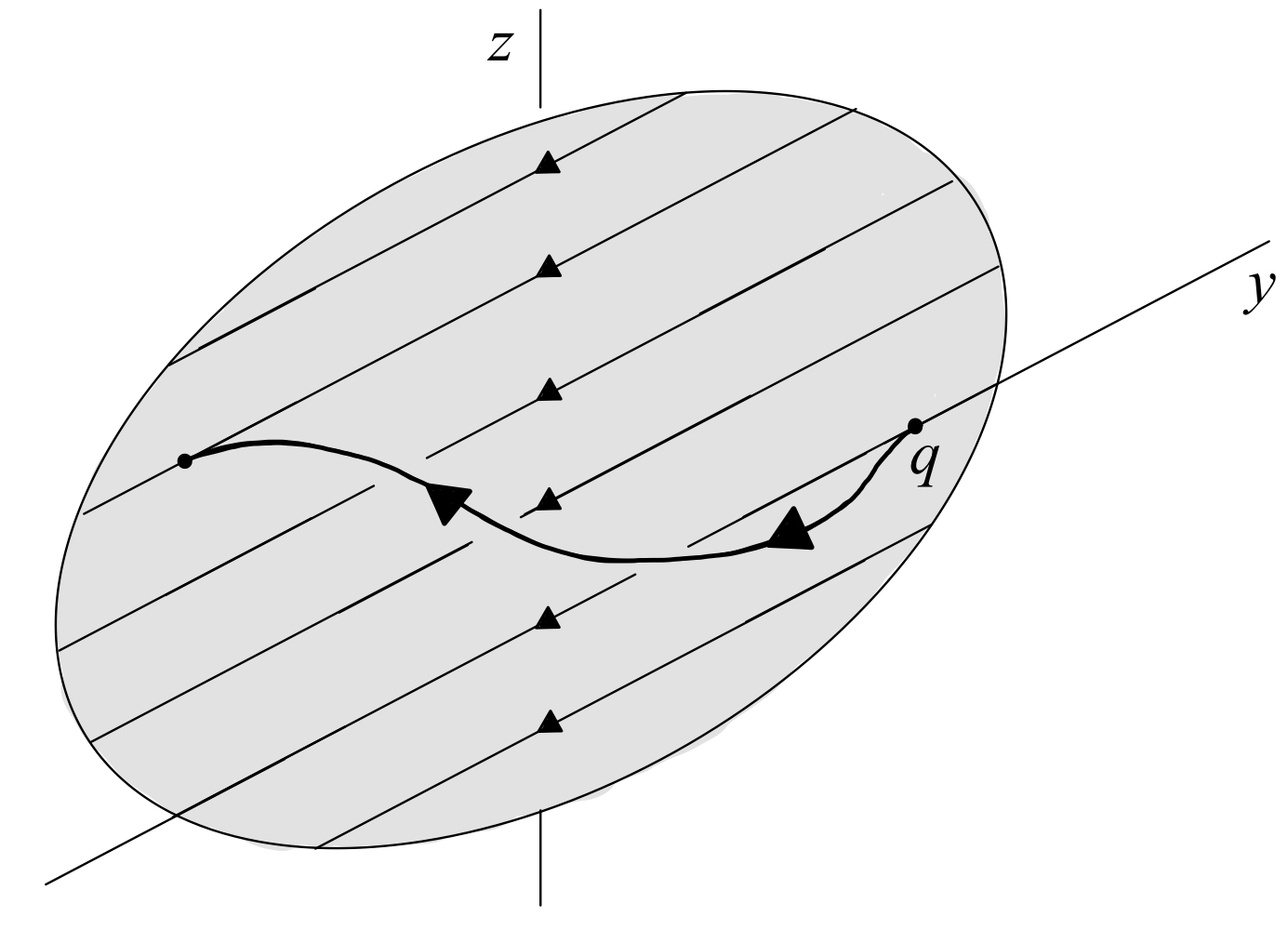} ~
	\includegraphics[width=0.35\linewidth]{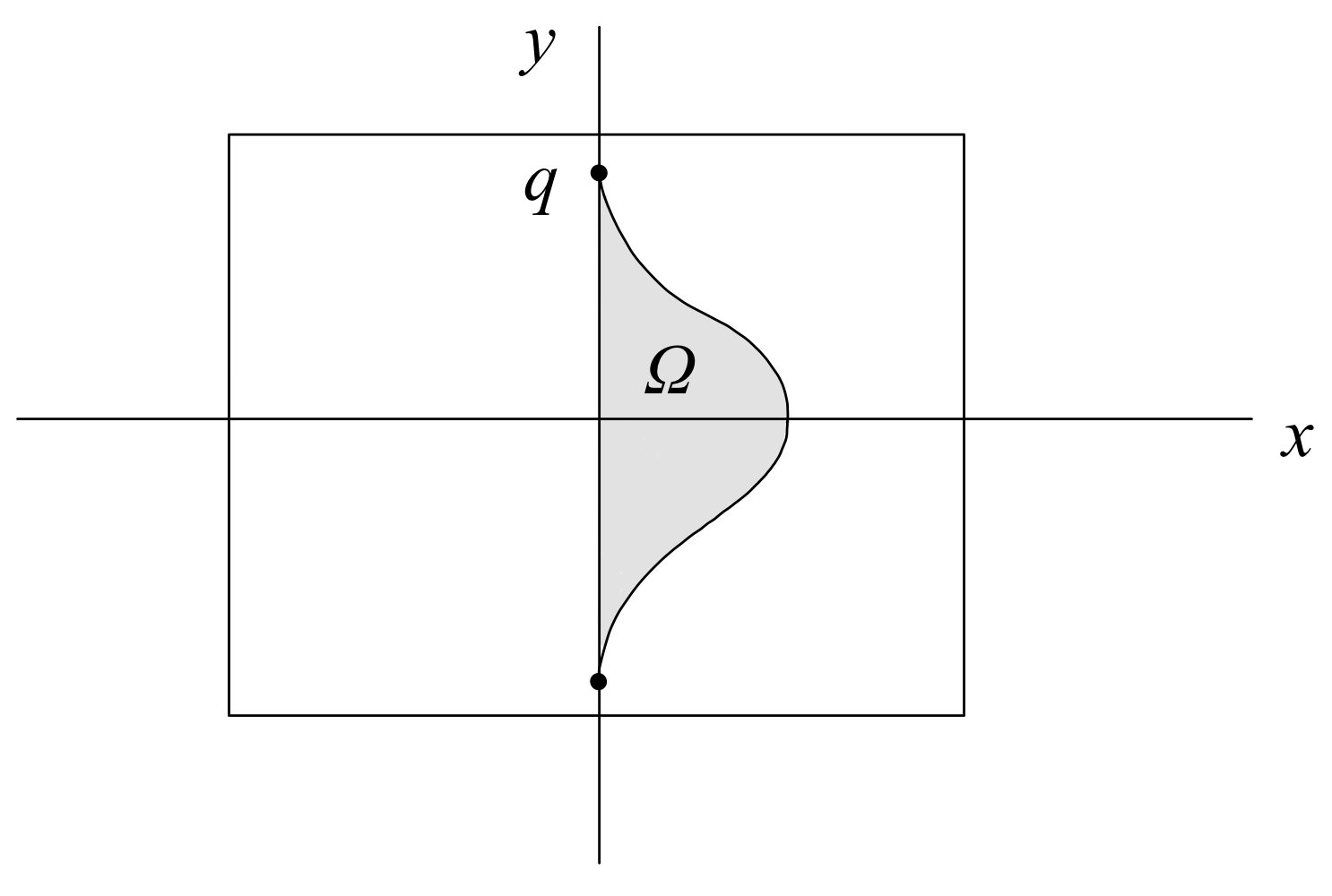}
	\caption{\small The lift to an horizontal curve connecting different leaves.}
	\label{fig:hyperbolic-perturbation}
\end{center}\end{figure}
 
The same argument has to be repeated \textit{mutatis mutandis} in a neighbourhood of $x_{2}$, ensuring that one connects $x_{2}$ exactly to the leaf coming from~$x_{1}$.
This is possible due to the continuity property of a hyperbolic sector, which ensures that the leaf coming from $x_{1}$ intersects the domain of the rectifying contactomorphism of~$x_{2}$.
\end{proof}

We can finally prove Theorem~\ref{thm:FiniteMetricSph}.  

\begin{proof}[Proof of Theorem~\ref{thm:FiniteMetricSph}]
The surface $S$ admits a global characteristic vector field, due to Lemma~\ref{lem:existance-cvf}.
Next, a surface with the topology of a sphere doesn't allow flows with nontrivial recurrent trajectories, see \cite[Lem.~2.4]{ABZintrod}. 
Indeed,  from a nontrivial recurrent trajectory one can construct a closed curve transversal to the flow which does not separate the surface, which contradicts the Jordan curve theorem. 

 Then, Lemma~\ref{lem:no-periodic-trajectories} and  Lemma~\ref{lem:no-sided-contours} imply that the flow of a characteristic vector field of $S$ does not contain periodic trajectories and sided contours, thus the hypothesis of Proposition~\ref{prop:dS-finite}  are satisfied. Consequently, $d_{S}$ is finite.
\end{proof}

%% %% %% %% %% %% %% %% %% %% %% %% %% %% %% %% %% %% %
%	SECTION 6 --- EXAMPLES
%% %% %% %% %% %% %% %% %% %% %% %% %% %% %% %% %% %% %

\section{Examples of surfaces in the Heisenberg  structure}\label{sect:examples}

In this section we present some examples of surfaces in the Heisenberg sub-Riemannian structure, that is the contact, tight, sub-Riemannian structure of $\R^{3}$ for which $(X_{1}, X_{2})$ is a global orthonormal frame, where 
\[
	X_{1}=\partial_{x}-y/2~\partial_{z}, \qquad X_{2}=\partial_{y}+x/2~\partial_{z}.
\]
If $(u,v)\mapsto (x(u,v),y(u,v),z(u,v))$ is a parametrisation of a surface $S$, then the characteristic vector field $X$ in coordinates $u,v$ becomes 
\begin{equation}\label{eq:XS-in-parametrisation}
	 X = 
	-\left(z_{v}+x_{v}\frac{y}{2} -y_{v}\frac{x}{2}\right) \frac{\partial }{\partial u} 
		+ \left( z_{u}+x_{u} \frac{y}{2}-y_{u}\frac{x}{2} \right) \frac{\partial }{\partial v},
\end{equation}
where have used the subscripts to denote a partial derivative. 
When the surface is the graph of a function $S=\{z=h(x,y)\}$, then in the graph coordinates
\[
 X= \left(  \frac{x}{2}-\partial_{y}h\right) 
			\frac{\partial}{\partial x} 
		+\left(\partial_{x}h +\frac{y}{2} \right)
			\frac{\partial}{\partial y},
\]
and, at a characteristic point $p=(x,y,z)$, the metric coefficient $\widehat K_{p}$ is computed by
\[
	\widehat K_{p}=-3/4+  \partial^{2}_{xx}h(x,y) ~\partial^{2}_{yy}h(x,y)-\partial^{2}_{xy}h(x,y)~ \partial^{2}_{yx}h(x,y).
\]

%% %% %% %% %% %% %% %% %% %% %% %% %% %% %% %% %% %% %

\subsection{Planes.}\label{exmp:planes}
Let us consider affine planes in Heisenberg. Thanks to the left-invariance, it is not restrictive to consider a plane $\mathcal{P}$ going throughout the origin. Thus,
\[
	\mathcal{P}=\{ (x,y,z)\in\R^{3} ~|~ax+by+cz=0\} \qquad \mbox{with } (a,b,c) \neq (0,0,0).
\] 

If $c= 0$, i.e., the plane is vertical, then $\mathcal{P}$ does not contain characteristic points. 
Every characteristic vector field is parallel to the vector $(b,-a,~0)$, therefore the characteristic foliation of $\mathcal{P}$ consists of lines that are parallel to the $xy$-plane. 
This implies that points with different $z$-coordinate are not at finite distance from each other, see Figure~\ref{Planes} (left). 
\medskip

Otherwise, if $c\neq0$, then $\mathcal{P}$ has exactly one characteristic point  $p=(-2b/c, 2a/c, 0)$. 
One has that 
\[
	\widehat K_{p} = -\frac{3}{4}.
\]
Thus, because of formula~(\ref{eq:eigenvalues-by-K-hat}), there is one eigenvalue of multiplicity two. 
Due to Corollary~\ref{cor:qualitative-char-foli-non-deg}, the characteristic foliation of $\mathcal{P}$ has a node  at $p$.
An explicit computation of $X_{S}$ shows that 
\[
	X_{S}(q) = \frac{q-p}{2} \qquad \forall\,q\in\mathcal{P},
\]
which shows that the characteristic foliation of $\mathcal{P}$ is composed of Euclidean half-lines radiating out of~$p$.
The metric $d_{\mathcal{P}}$ induced by the Heisenberg group on $\mathcal{P}$ satisfies the following relation: for all $q,q'\in\mathcal{P}$, one has 
\[
	d_{\mathcal{P}}\big(q,q'\big) = 
		\begin{cases} 
		 |(x,y)-(x',y')|_{\R^{2}},  & \mbox{if}~ (q-p) \sslash  (q'-p) \\
		d_{\mathcal{P}}(q,p) + d_{\mathcal{P}}(q',p), &  \mbox{otherwise,}  \end{cases} 
\]
where we have written $q=(x,y,z)$ and $q'=(x',y',z')$. This distance is sometimes called British Rail metric. See Figure~\ref{Planes} (right).

\begin{figure}[] \begin{center}
	\includegraphics[width=0.25\textwidth]{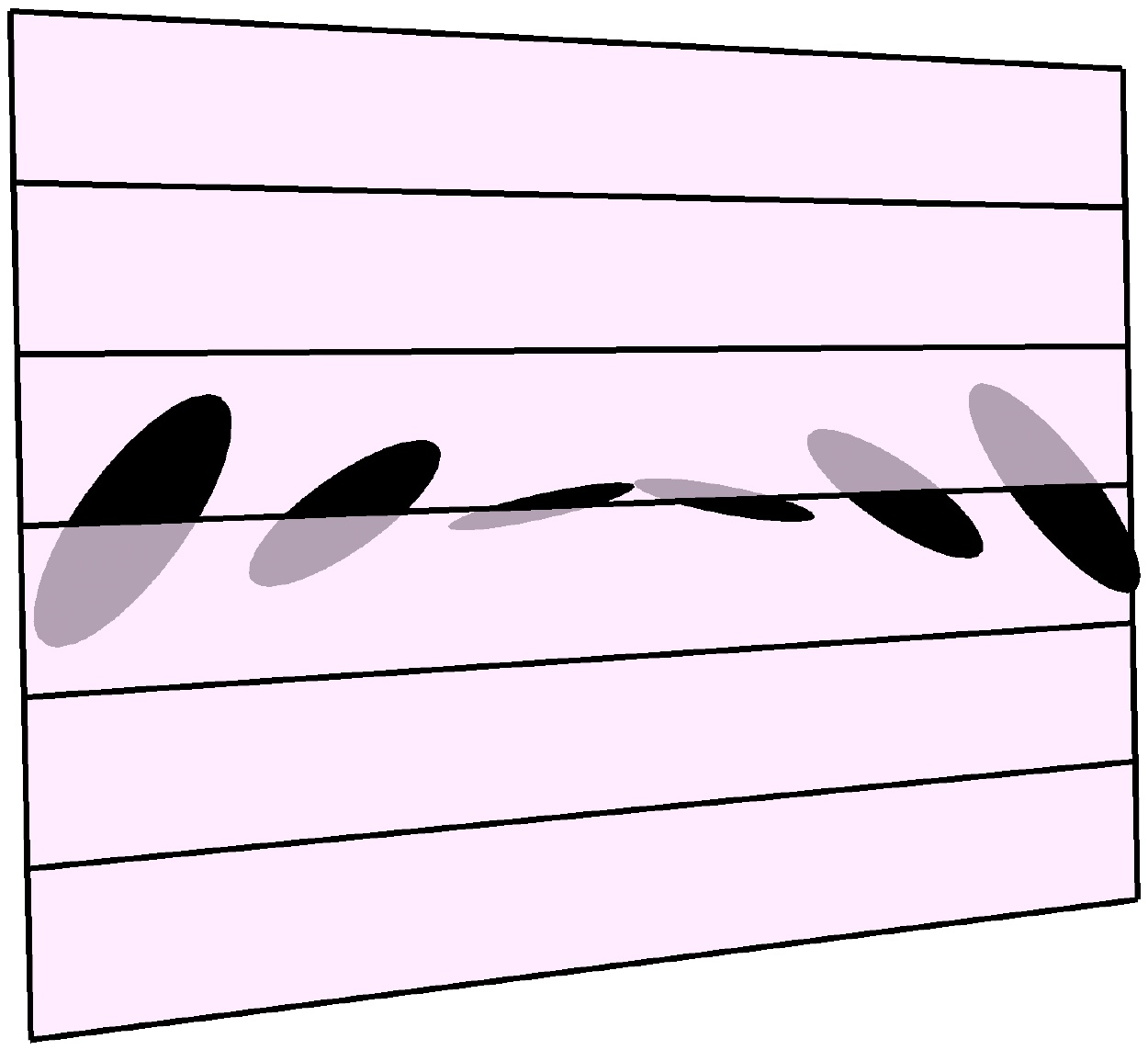} \hspace{30pt}
	\includegraphics[width=0.35\textwidth]{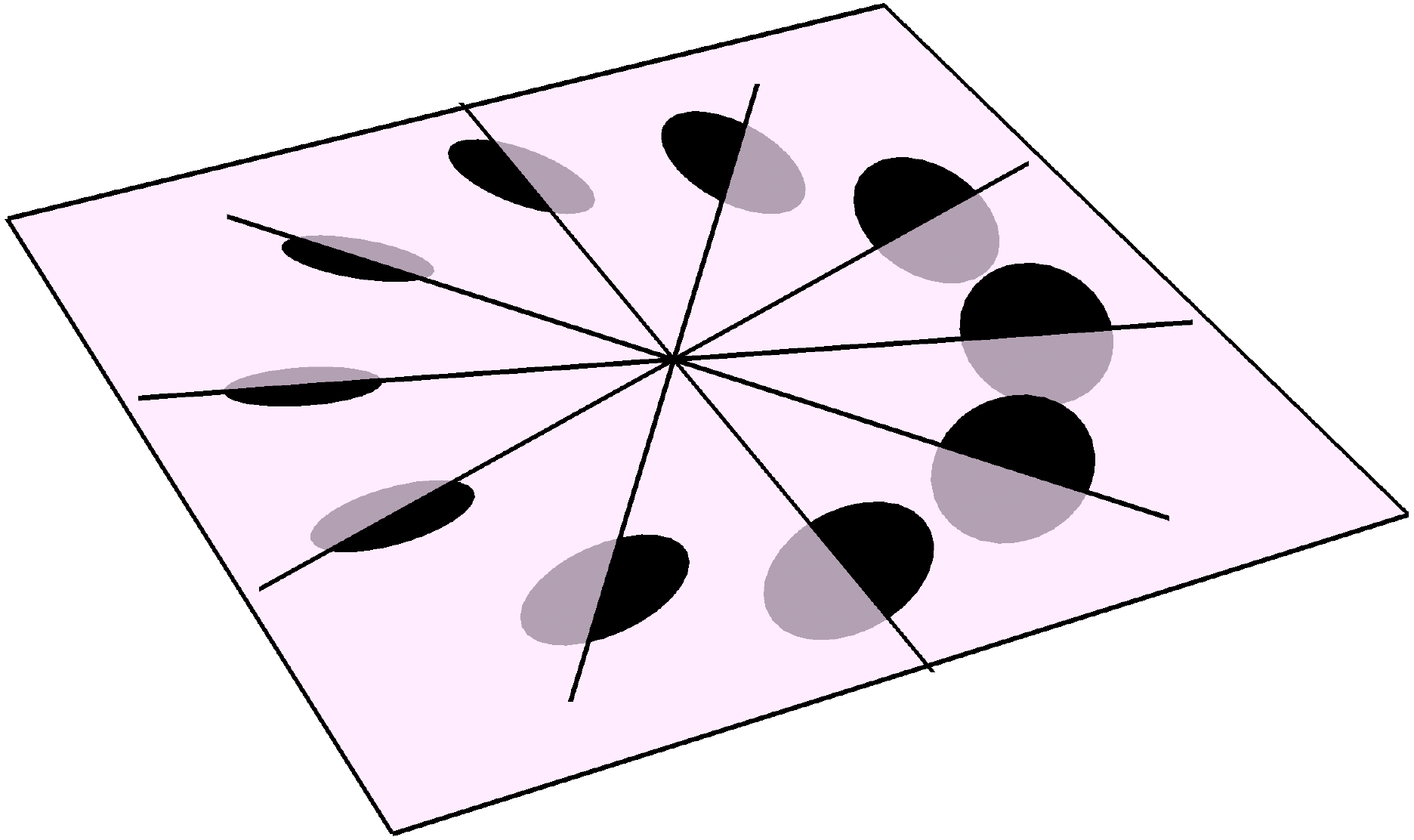}
	\caption{ \label{Planes}\small The qualitative picture of the characteristic foliation of a vertical plane (left), and of a non-vertical plane (right).}
\end{center}\end{figure}

%% %% %% %% %% %% %% %% %% %% %% %% %% %% %% %% %% %% %

\subsection{Ellipsoids} 
	 Fix $a,b,c>0$, and consider the surface $\mathcal{E}=\mathcal{E}_{a,b,c}$ defined by 
\[
	\mathcal{E}_{a,b,c}=\bigg\{ (x,y,z)\in \R^{3} ~\Big| ~ \frac{x^{2}}{a^{2}}+\frac{y^{2}}{b^{2}}+\frac{z^{2}}{c^{2}}-1=0\bigg\}.
\] 
This surface has exactly two characteristic points $p_{1}=(0,0,c)$ and $p_{2}=(0,0,-c)$, respectively at the North  and the South pole. 
For both points, one has
\[
	\widehat K_{p_{i}} = -\frac{3}{4} +\frac{c^{2}}{a^{2}b^{2}}, \qquad i=1,2.
\]
Because of Corollary~\ref{cor:qualitative-char-foli-non-deg}, the characteristic foliation of $\mathcal{E}$ spirals around the two poles,   as in Figure~\ref{img:sphere}. 
Due to Proposition~\ref{prop:leavesFiniteLength}, the spirals converging to the poles have finite sub-Riemannian length, thus the length distance $d_{\mathcal{\mathcal{S}}}$ is finite. 
Indeed, $d_{\mathcal{S}}$ is realised by the length of the curves joining the points with either the North, or the South pole. 
Here, the finiteness of $d_{S}$ is also a particular case of Theorem~\ref{thm:FiniteMetricSph}.

%% %% %% %% %% %% %% %% %% %% %% %% %% %% %% %% %% %% %

\subsection{Symmetric paraboloids}

Let  $a\in\R $, and consider the paraboloid $\mathcal{P}_{a}$ with
\[
	\mathcal{P}_{a}=\big\{ (x,y,z)\in\R^{3} ~|~z=a\big(x^{2}+y^{2}\big)\big\}.
\] 
The origin $p$ is the unique characteristic point of $\mathcal{P}_{a}$. Note that 
\[
	\widehat K_{p}=-\frac{3}{4}+4a^{4},
\]
therefore the characteristic foliation  is a focus. 

%% %% %% %% %% %% %% %% %% %% %% %% %% %% %% %% %% %% %

\subsection{Horizontal torus}\label{exmp:horizontal-torus}
Fix $R>r>0$, and consider the torus   parametrised by
\[
	\Phi(u,v)=\big( (R+r\cos u)\cos v, (R+r \cos u)\sin v, r\sin u \big). 
\]
This is the torus obtained by revolving a circle of radius $r>0$ in the $xz$-plane around a circle of radius $R>r$ surrounding the $z$-axis. 
Using formula~(\ref{eq:XS-in-parametrisation}), a  characteristic vector field $X$ in the coordinates $(u,v)$ is \begin{equation} \label{eq:XS-horiz-torus}
	X=	\frac{\left(R + r \cos(u) \right)^{2}}{2} \frac{\partial }{\partial u} - \frac{ r \cos(u) }{2} \frac{\partial }{\partial v}.
\end{equation}
  It is immediate to see that the characteristic set  is empty. 
Thus, no point can be a limit point of any leaves of the characteristic foliation; due to Remark~\ref{rmq:finite-length-implies-convergence}, this implies that  the length distance is infinite.

\begin{figure}[]
\begin{center}
	\includegraphics[width=0.45\textwidth]{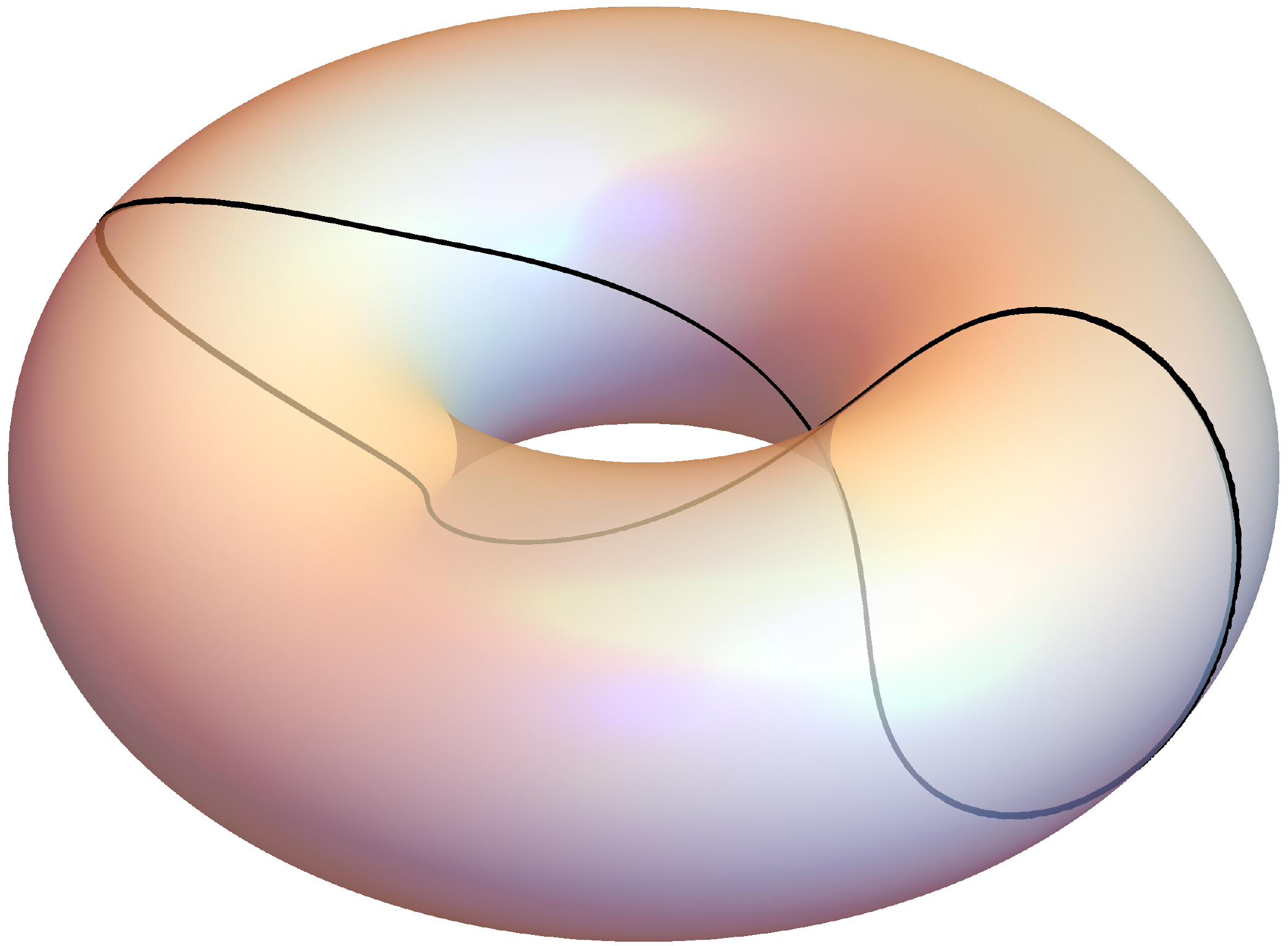} ~ ~ % r = 0.5437 & R = 1
	\includegraphics[width=0.45\textwidth]{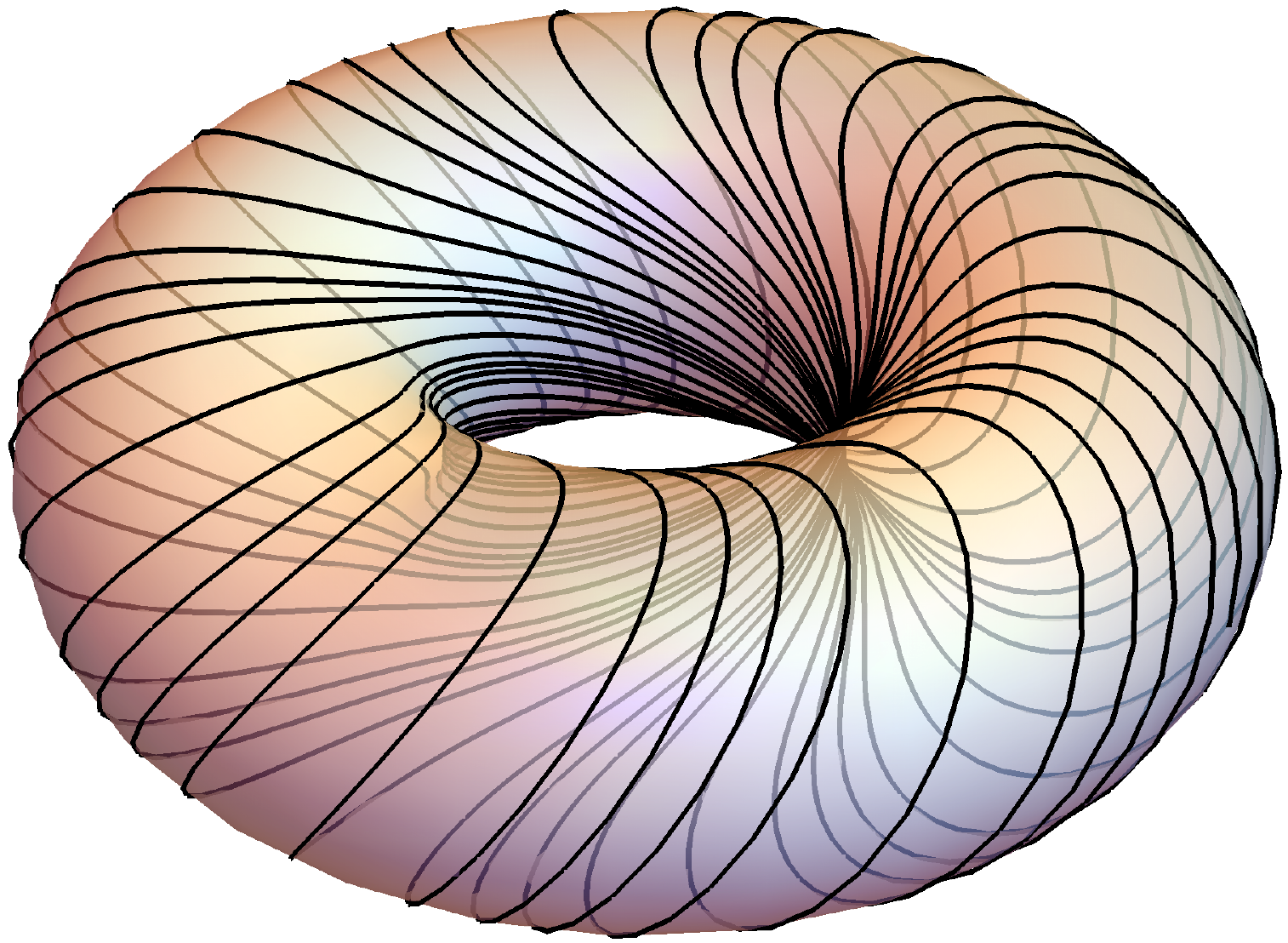}
	\caption{ \label{horTor} \small A leaf of the characteristic foliation of two Horizontal tori. On the left-hand side the leaf is periodic, and on the right-hand side there is a portion of an everywhere dense leaf .  }
\end{center}\end{figure}

\begin{lemma}	The characteristic foliation of a horizontal torus is filled either with   periodic trajectories, or  with everywhere dense trajectories.
	\end{lemma}
	\begin{proof}
Using expression~(\ref{eq:XS-horiz-torus}), in the coordinates $u,v$  the trajectories of $X$ satisfy
	\begin{equation}\label{eq:char-foliat-horiz-torus}
	\left\{
	\begin{array}{l} \dot u = \left(R + r \cos(u) \right)^{2}/2 \\ \dot v = -  r \cos(u) /2 .\end{array}
	\right.
	\end{equation}
Because the Heisenberg distribution and the horizontal torus are invariant under rotations around the $z$-axis, the same applies to the characteristic  foliation.
Thus, the solutions of (\ref{eq:char-foliat-horiz-torus}) are  $v$-translations of the solution $\gamma_{0}(t)=(u(t),v(t))$ with initial condition $\gamma_{0}(0)=(0,0)$.

Note that $(r+R)^{2}/2 \geq \dot u (t)\geq(R-r)^{2}/2 $. Thus, there exists a time $t_{0}$ in which the trajectory $\gamma_{0}(t)$,  satisfies $u(t_{0})=2\pi$. 
Define $\alpha_{r,R}=v(t_{0})$.
If $\alpha_{r,R}/(2\pi)=m/n$ is rational, then $\gamma_{0}(nt_{0})=0\pmod {2\pi}$. This shows that $\gamma_{0}(t)$ is periodic, as every other trajectory.
On the other hand, if $\alpha_{r,R}/(2\pi)$ is irrational, then a classical argument shows that  $\gamma(t)$ is dense in the torus, see  for instance \cite[E.g~.2.3.1]{ABZintrod}.
\end{proof}

See Figure~\ref{horTor} for a picture of a leaf in these two cases. 

%% %% %% %% %% %% %% %% %% %% %% %% %% %% %% %% %% %% %

\subsection{Vertical torus}\label{exmp:vertical-torus}
Fix $R>r>0$, and consider the torus   $\mathcal{T}=\mathcal{T}_{r,R}$  parametrised by
\[
	\Phi(u,v)=\big(r \sin u, (R+r\cos u)\cos v , (R+r\cos u)\sin v \big).
\]
This is the torus obtained by turning a circle of radius $r$ in the $xy$-plane around a circle of radius $R$ surrounding the $x$-axis.
Due to formula~(\ref{eq:XS-in-parametrisation}), a  characteristic vector field $X$ in coordinates $u,v$ is
\begin{align*}
	X =& (R+r\cos u)\Big(\cos v + \frac{r}{2} \sin v \sin u\Big) \frac{\partial}{\partial u} \\
		&  + \frac{r}{2}\Big( 2 \sin u \sin v - R \cos u \cos v - r \cos v\Big)\frac{\partial}{\partial v}.
\end{align*}
The characteristic points are critical points of the vector field $X$.
If $\cos v = \sin u =0$, then $(u,v)$ corresponds to a solution; this gives 4 characteristic points 
\[
	F_{\pm} = \big(0,0,\pm(R+r) \big), \qquad V_{\pm} =  \big(0,0,\pm(R-r) \big).
\]
The other critical points of $X$ occur if and only if 
	\begin{equation}\label{eq:additional-c-p}
		\tan v =-\frac{2}{r\sin u}, \qquad \cos u = -\frac{4+r^{2}}{rR}.
	\end{equation}
System \eqref{eq:additional-c-p} has solutions if and only if $R> 4$ and $|2r-R|\leq \sqrt{R^{2}-16}$, in which case we have 4 additional characteristic points $S_{i}(r,R)$, for $i=1,2,3,4$.
 Now, the metric coefficient at the characteristic points $F_{\pm}$ and $V_{\pm}$ is
 \begin{align*}
 	\widehat K_{F_{\pm}} &= -\frac{3}{4}+\frac{1}{r(R+r)}, \\
	\widehat K_{V_{\pm}} & =-\frac{3}{4}-\frac{1}{r(R-r)}.
 \end{align*}
 Note that $\widehat K_{F_{\pm}}>-3/4$, thus, due to Corollary~\ref{cor:qualitative-char-foli-non-deg}, $F_{\pm}$ is a focus for all value of $r$ and $R$. 
On the other hand, $\widehat K_{V_{\pm}}$ can attain any value between $-\infty$ and $-3/4$; precisely:
\begin{itemize}[-]
\item if $R< 4$ or $|2r-R|> \sqrt{R^{2}-16}$,  then $\widehat K_{V_{\pm}}<-1$ and $V_{\pm}$ are saddles.
\item if $|2r-R|= \sqrt{R^{2}-16}$, then $\widehat K_{V_{\pm}}=-1$ and $V_{\pm}$ is a degenerate characteristic point; due to the Poincaré Index theorem, the points $V_{\pm}$ are  saddles.
\item if $|2r-R|< \sqrt{R^{2}-16}$, then $-1<\widehat K_{V_{\pm}}<-3/4$ and $V_{\pm}$ are nodes.
\end{itemize}
The values for which $|2r-R|= \sqrt{R^{2}-16}$ are a bifurcation of the dynamical system $X$, because the number of characteristic point changes from 4 to 8. 
The  characteristic points $S_{i}$ which appears at this bifurcation are saddles, due to the Poincaré Index theorem. 
The bifurcation which takes place is the one presented in \cite[E.g.~4.2.6]{perko2012differential}.
	
\begin{figure}[]
\begin{center}
	\includegraphics[align=c, width=0.45\textwidth]{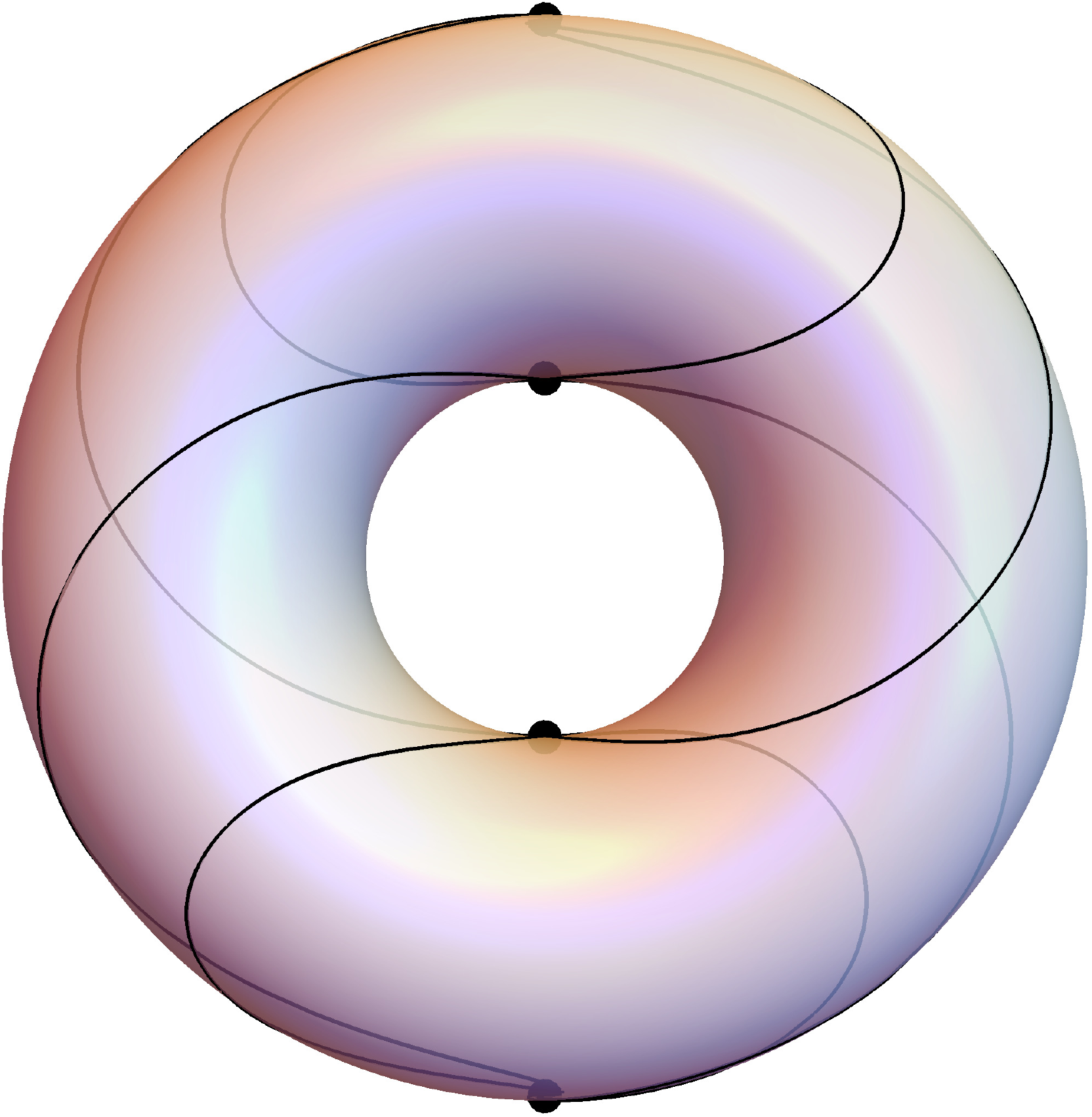} ~ ~
	\includegraphics[align=c, width=0.45\textwidth]{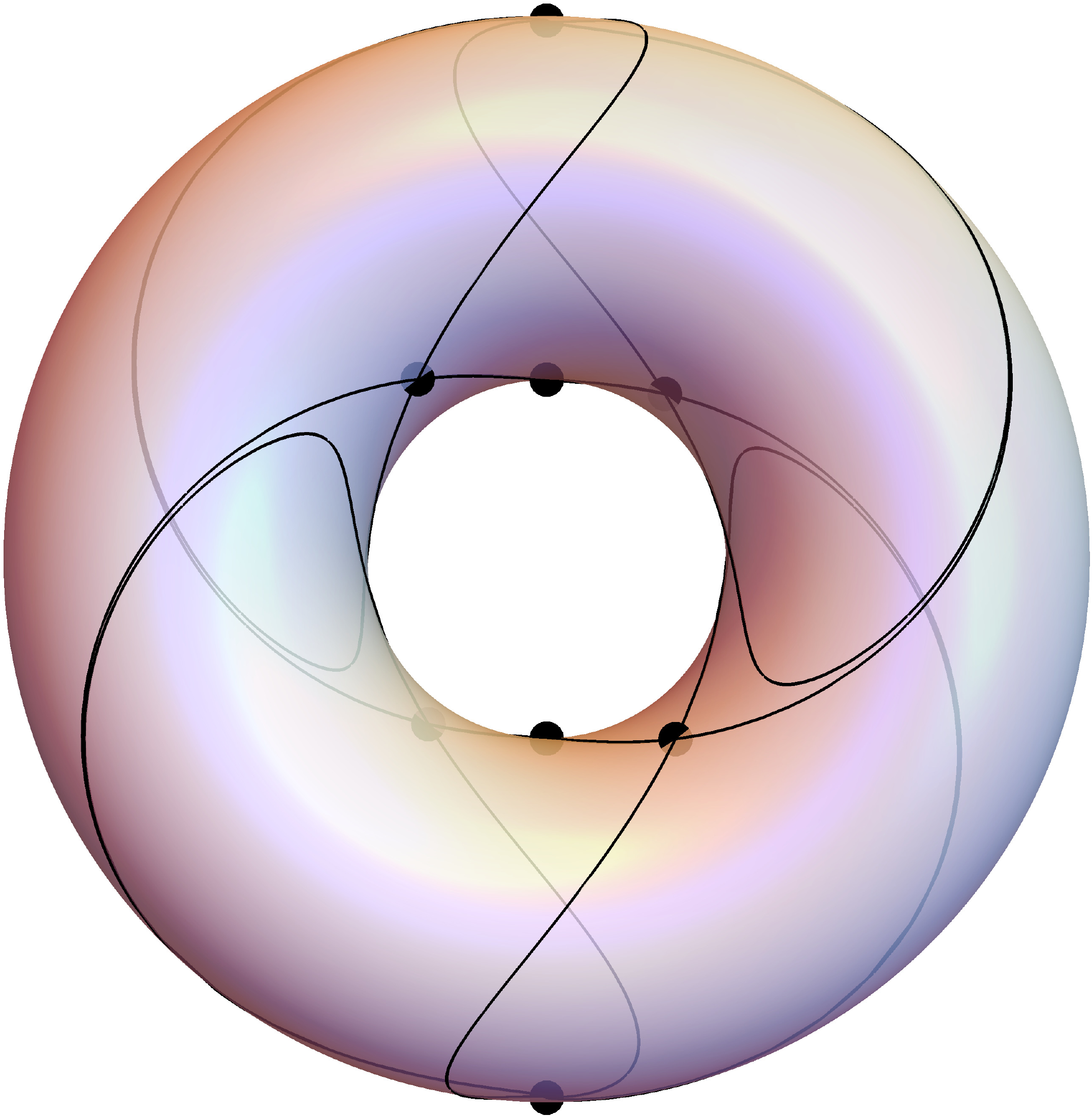}
	\caption{\label{verTor}\small The topological skeleton, i.e., the singular trajectories, of the characteristic foliations of two vertical tori: the torus on the left-hand side has four characteristic points, and the torus on the right-hand side has eight. }
\end{center}\end{figure}

%% %% %% %% %% %% %% %% %% %% %% %% %% %% %% %% %% %% %
%	APPENDIX
%% %% %% %% %% %% %% %% %% %% %% %% %% %% %% %% %% %% %

\appendix

%% %% %% %% %% %% %% %% %% %% %% %% %% %% %% %% %% %% %

\section{On the center manifold theorem}\label{append:center-manif}

In the language of dynamical systems, a non-degenerate characteristic point $p$ is a \textit{hyperbolic equilibrium} for any characteristic vector field $X$, i.e., an equilibrium for which the real parts of the eigenvalues of $DX(p)$ are non-zero.
	For a hyperbolic equilibrium $p$, the Hartman-Grobman theorem and the Hartman theorem give a conjugation between the flow of $X$ and the flow of $DX(p)$, see \cite[Par.~2.8]{perko2012differential} and   \cite{Hartman1960OnLH}.
	\\
	
	Let us discuss here the case of a non-hyperbolic equilibrium, i.e., of a degenerate characteristic point.
Let $E$ be an open set of $\R^{n}$ containing the origin, and let $X$ be a vector field in $C^{1}(E,\R^{n})$ with $X(0)=0$. 
Due to the Jordan decomposition theorem,  we can assume that the linearisation of $X$ at the origin is 
\[
	DX(0)=
		\left(
		\begin{array}{ccc}
		C & & \\
		 & P & \\
		& & Q
		\end{array}
		\right),
\]		
where $C$ is a square $c\times c$ matrix with $c$ complex (generalised) eigenvalues with zero real part, $P$ with $p$ complex (generalised) eigenvalues with positive real part, and $Q$ with $q$ complex (generalised) eigenvalues with negative real part.
Thus, the dynamical system $\dot\gamma=X(\gamma)$ can be rewritten as
\[
	\left\{
	\begin{array}{l}
		\dot x= Cx+F(x,y,z) \\
		\dot y= Py+G(x,y,z) \\
		\dot z= Qz+H(x,y,z) 
		\end{array}
	\right.
\]
for $(x,y,z)\in \R^{c}\times\R^{p}\times\R^{q}=\R^{n}$, and for suitable functions $F$, $G$ and $H$ with $F(0)=G(0)=H(0)=0$ and  $DF(0)=DG(0)=DH(0)=0$.
	\\

The origin is a non-hyperbolic characteristic point if and only if $c\geq 1$. 
Under these hypotheses, the following theorem shows that there exists an embedded submanifold $\mathcal{C}$ of dimension $c$, tangent to $\R^{c}$, and invariant for the flow of $X$. Such manifold is called a \textit{central manifold} of $X$ at the origin.

\begin{proposition} [{\cite[Par.~2.12]{perko2012differential}}] \label{CenterManifold}
	Under the previous  notations,  there exists an open set $U\subset \R^{c}$ containing the origin, and two functions $h_{1}:U\to \R^{p}$ and $h_{2}:U\to \R^{q}$ of class  $C^{1}$ with $h_{1}(0)=h_{2}(0)= 0$ and $Dh_{1}(0)=Dh_{2}(0)=0$, and such that the map $ x \mapsto (x,h_{1}(x), h_{2}(x))$ parametrises a submanifold invariant for the flow of $X$. 
	Moreover, the flow of $X$ is $C^{0}$-conjugate to the flow of
\begin{equation}\label{eq:center-manifold-dynamic}
	\left\{ \begin{array}{l}
	\dot x = C x + F(x, h_{1}(x), h_{2}(x)) \\
	\dot y = P y \\
	\dot z = Q z.
	\end{array}\right.
	\end{equation}
\end{proposition}

In general, the central manifold $\mathcal{C}$ is non-unique. Note that the dynamic of the $x$-variable in equation (\ref{eq:center-manifold-dynamic}) is simply the restriction of $X$ to the center manifold $\mathcal{C}$. 
One can show that the trajectory converging to the origin approaches $\mathcal{C}$ exponentially fast: this is the asymptotic approximation property we used in~(\ref{eq:asyn-approx-property}).
	
\begin{proposition} [{\cite[p.~330]{AlbBress07}}] \label{approxCM}
	  Under the previous assumptions, let us denote $\mathcal{C}$ a center manifold of the flow of $X$ at the origin. 
	Then, for   every trajectory $ l(t)$  such that $l(t)\to 0 $ as $t \to +\infty$, there exists $\eta >0$ and a trajectory $\zeta(t)$ in the center manifold $\mathcal{C}$, such that
	\begin{equation*}
	e^{\eta t}|l(t)-\zeta(t)|_{\R^{n}} \to 0, \qquad as \quad  t \to +\infty.
\end{equation*}
\end{proposition}

%% %% %% %% %% %% %% %% %% %% %% %% %% %% %% %% %% %% %

\section{Tight and overtwisted distributions}\label{apx:tight}
In this section we briefly recall the theory of tight distributions. For a more comprehensive presentation, we refer to \cite[Par.~4.5]{geiges2008introduction}. In what follows $M$ is a 3-dimensional contact manifold, whose distribution is $D$.
	\\

To define an overtwisted disk, let us first consider an embedding of $\Delta=\{x\in \R^{2}:   |x|\leq1\}$ in  $M$, and denote $\Gamma=\partial \Delta$.
Let $\Gamma$ be horizontal with respect to the contact distribution $D$, i.e., $T\Gamma\subset D$. 
Then,  the normal bundle $N\Gamma=TM|_{\Gamma} / T\Gamma$ can be decomposed in two ways: the first with respect to the tangent space of $\Delta$, i.e.,
\begin{equation} \label{eq:surface-splitting}
	N\Gamma \cong \faktor{TM }{T\Delta} \oplus \faktor{T\Delta }{T\Gamma},
\end{equation}
and the second with respect to the contact distribution $D$, i.e.,
\begin{equation}\label{eq:distribution-splitting}
	N\Gamma \cong \faktor{TM }{D} \oplus \faktor{D }{T\Gamma}.
\end{equation}
A frame $(Y_{1},Y_{2})$ of $N\Gamma$ is called a \textit{surface frame} if it respects the splitting~(\ref{eq:surface-splitting}), i.e., $Y_{1} \in TM\mysetminus T\Delta$ and $Y_{2}\in T\Delta\mysetminus T\Gamma$; similarly, it is called a \textit{contact frame} if it respects the splitting~(\ref{eq:distribution-splitting}). 
Since the contact distribution is cooriented near $\Delta$, both bundles (\ref{eq:surface-splitting}) and (\ref{eq:distribution-splitting}) are trivial, thus one can always find  contact and surface frames.

 The Thurston--Bennequin invariant of $\Gamma$, noted $\tb(\Gamma)$, is the number of twists of a contact frame of $\Gamma$ with respect to a surface frame: the right-handed twists are counted positively, and the left-handed twists negatively (cf. for instance~\cite[Def.~3.5.4]{geiges2008introduction}). 
 Note that $\tb(\Gamma)$ is independent of the orientation of $\Gamma$.
 The requirement that the distribution $D$ does not twist along the boundary of $\Delta$ is equivalent to $\tb(\partial\Delta)=0$, i.e., the Thurston--Bennequin invariant of $\partial \Delta$ being zero. 
 
\begin{definition} \label{defn:overtwisted}
	 An embedded disk $\Delta$ in a cooriented contact manifold $(M,D)$ with smooth boundary $\partial \Delta$ is an \textit{overtwisted disk} if $\partial \Delta$ is a horizontal curve of $D$, $\tb(\partial\Delta)=0$, and there is exactly  one characteristic point in the interior of the disk.
\end{definition}

Note that the elimination lemma of Giroux allows to remove the condition that there is only one characteristic point in the interior of the owertwisted disk, as discussed for instance in \cite[Def.~4.5.2]{geiges2008introduction}.

\begin{definition} \label{defn:tight}
A contact structure $(M, D)$ is called  \textit{overtwisted} if it admits an overtwisted disk, and \textit{tight} otherwise.
\end{definition}

%% %% %% %% %% %% %% %% %% %% %% %% %% %% %% %% %% %% %
%	REFERENCES
%% %% %% %% %% %% %% %% %% %% %% %% %% %% %% %% %% %% %

\bibliographystyle{alphaabbrv}
\bibliography{biblio-13.bib}

%% %% %% %% %% %% %% %% %% %% %% %% %% %% %% %% %% %% %

\end{document}